\documentclass[12pt]{amsart}
\pdfoutput=1  
\usepackage{amssymb,amsmath,latexsym,amsthm}
\usepackage{mathpazo}
\usepackage[colorlinks, linkcolor=BlueViolet, urlcolor=BlueViolet, citecolor=BlueViolet]{hyperref}
\usepackage[margin=0.99in]{geometry}      
\usepackage[usenames,dvipsnames]{xcolor} 
\usepackage{colonequals}                 
\usepackage{booktabs}                    
\usepackage{caption}                     
\usepackage[all]{xy}

\newtheorem{theorem}{Theorem}[section]
\newtheorem{lemma}[theorem]{Lemma}
\newtheorem{proposition}[theorem]{Proposition}
\newtheorem{corollary}[theorem]{Corollary}
\newtheorem{heuristic}[theorem]{Heuristic}
\newtheorem{fheuristic}[theorem]{Na\"{\i}ve Heuristic}

\theoremstyle{definition}
\newtheorem{remark}[theorem]{Remark}
\newtheorem{notation}[theorem]{Notation}
\newtheorem{definition}[theorem]{Definition}
\newtheorem{example}[theorem]{Example}

\numberwithin{equation}{section}

\DeclareMathOperator{\Aut}{Aut}

\DeclareMathOperator{\End}{End}
\DeclareMathOperator{\Gal}{Gal}
  
\DeclareMathOperator{\Jac}{Jac}
\global\let\ker\undefined  \DeclareMathOperator{\ker}{Ker}
\DeclareMathOperator{\Prym}{Prym}
\DeclareMathOperator{\trace}{trace}

\newcommand \bbA {{\mathbb A}}
\newcommand \EE {{\mathbb E}}
\newcommand \FF {{\mathbb F}}
\newcommand \Fq {\FF_q}
\newcommand \PP {{\mathbb P}^1}
\newcommand \QQ {{\mathbb Q}}
\newcommand \ZZ {{\mathbb Z}}

\newcommand \CA {{\mathcal A}}
\newcommand \CX {{\mathcal X}}
\newcommand \CZ {{\mathcal Z}}

\newcommand{\frakp}{{\mathfrak p}}

\newcommand{\alphabar}{\overline{\alpha}}
\newcommand{\betabar}{\overline{\beta}}
\newcommand{\Kbar}{\bar{K}}
\newcommand{\CZbar}{\CZ\llap{$\overline{\phantom{\mathrm Z}}$}} 

\newcommand{\Ctilde}{\tilde{C}}  
\newcommand{\Ztilde}{\tilde{Z}}  

\newcommand{\Zhat}{\hat{Z}}  

\newcommand{\pz}{\phantom{0}}   

\newcommand{\mybar}[1]{
  \mathchoice
  {#1\llap{$\overline{\phantom{\displaystyle\mathrm#1}}$}}
  {#1\llap{$\overline{\phantom{\textstyle\mathrm#1}}$}}
  {#1\llap{$\overline{\phantom{\scriptstyle\mathrm#1}}$}}
  {#1\llap{$\overline{\phantom{\scriptscriptstyle\mathrm#1}}$}}
}  
\renewcommand{\bar}{\mybar}
\renewcommand{\tilde}{\widetilde}
\renewcommand{\hat}{\widehat}

\makeatletter
\@namedef{subjclassname@2020}{%
  \textup{2020} Mathematics Subject Classification}
\makeatother

\title{Doubly isogenous genus-2 curves with $D_4$-action}
\date{26 January 2023}

\author[Arul]{Vishal Arul}
\address[Arul]{MIT Department of Mathematics, 
         77 Massachusetts Ave., Bldg. 2-239A, 
         Cambridge, MA 02139, USA}
\email{varul.math@gmail.com}

\author[Booher]{Jeremy Booher}
\address[Booher]{School of Mathematics and Statistics, 
         University of Canterbury, Private Bag 4800,
         Christchurch 8140, New Zealand}
\email{jeremy.booher@canterbury.ac.nz}

\author[Groen]{Steven R. Groen}
\address[Groen]{Department of Mathematics, 
         University of Warwick, Zeeman Building,
         Coventry, CV4 7AL, UK}
\email{steven.groen@warwick.ac.uk}

\author[Howe]{Everett W. Howe}
\address[Howe]{Unaffiliated mathematician, 
         San Diego, CA 92104, USA}
\email{however@alumni.caltech.edu}

\author[Li]{Wanlin Li}
\address[Li]{Centre de recherches math\'ematiques,
         Universit\'e de Montr\'eal, 2920 Chemin de la tour,
         Montr\'eal (Qu\'ebec) H3T 1J4, Canada}
\email{liwanlin@crm.umontreal.ca}

\author[Matei]{Vlad Matei}
\address[Matei]{Raymond and Beverly Sackler School of Mathematical Sciences, 
         Tel Aviv University, 
         Tel Aviv 69978, Israel}
\email{vladmatei@mail.tau.ac.il}

\author[Pries]{Rachel Pries}
\address[Pries]{Department of Mathematics, 
         Colorado State University, 
         Fort Collins, CO 80523, USA}
\email{pries@math.colostate.edu}

\author[Springer]{Caleb Springer}
\address[Springer]{Department of Mathematics, 
         The Pennsylvania State University, 
         University Park, PA 16802, USA}
\email{cks5320@psu.edu}

\begin{document}

\begin{abstract}
We study the extent to which curves over finite fields are characterized by 
their zeta functions and the zeta functions of certain of their covers. Suppose
$C$ and $C'$ are curves over a finite field $K$, with $K$-rational base points 
$P$ and $P'$, and let $D$ and $D'$ be the pullbacks (via the Abel--Jacobi map) 
of the multiplication-by-$2$ maps on their Jacobians. We say that $(C,P)$ and 
$(C',P')$ are \emph{doubly isogenous} if $\Jac(C)$ and $\Jac(C')$ are isogenous
over $K$ and $\Jac(D)$ and $\Jac(D')$ are isogenous over $K$. For curves of 
genus $2$ whose automorphism groups contain the dihedral group of order eight, 
we show that the number of pairs of doubly isogenous curves is larger than 
na\"{\i}ve heuristics predict, and we provide an explanation for this 
phenomenon.
\end{abstract}



\maketitle

\section{Introduction}

The isogeny class of the Jacobian of a (smooth, projective, geometrically
connected) curve over a field $K$ is an invariant of the curve, and it is 
natural to wonder whether this invariant is strong enough to always distinguish
two curves from one another. The answer is no --- distinct isogenous elliptic
curves provide abundant counterexamples.  Even when we restrict attention to 
curves of larger genus, the answer remains no, as over finite fields curves that
are Galois conjugates of one another have the same zeta function (which in this 
case characterizes the isogeny class of the Jacobian). Furthermore, 
Smith~\cite{Smith2011} showed that even in characteristic $0$ there exist 
non-isomorphic curves $C$ and $C'$ of arbitrarily large genus with $\Jac(C)$ 
isogenous to $\Jac(C')$. Later, Mestre~\cite{Mestre2009, Mestre2013} proved that
such pairs of curves exist for \emph{every} genus in characteristic~$0$; in
particular, for every $g \ge 1$ there is a family of dimension $g+1$ of pairs of
hyperelliptic curves of genus $g$ with a $2$-power isogeny between their
Jacobians. Thus the isogeny class of the Jacobian is not an invariant that can
always distinguish two curves from one another.

The question becomes more interesting when we restrict to low-dimensional 
families of curves. One motivation is a connection with deterministic algorithms
for factoring polynomials over finite fields. Suppose there is an open subset
$U$ of $\bbA^1_\ZZ$ and an abelian scheme $\CA$ over $U$ such that for all 
sufficiently large primes~$p$, the zeta functions of the specialization of $\CA$
to the elements of $U(\FF_p)$ are distinct. Poonen~\cite{Poonen2019}, extending 
earlier work of Kayal, showed that if such a scheme $\CA$ exists, then there is
a deterministic algorithm that, given a finite field $\FF_q$ and a polynomial
$f\in \FF_q[t]$, will produce the irreducible factors of $f$ in time polynomial
in $\log q$ and $\deg f$.

Motivated by this observation, Sutherland and Voloch~\cite{SutherlandVoloch2019}
considered curves over finite fields, and asked whether curves could be 
distinguished from one another (up to Galois conjugacy) by their zeta functions
together with the zeta functions of certain of their covers. If so, then the
Jacobians of these curves and their covers could be used in Poonen's argument. 
One of the families of covers they studied was obtained by considering maximal 
unramified abelian $2$-extensions of curves, as will now describe.

Let $(C,P)$ be a \emph{pointed} curve over a field $K$, that is, a curve over 
$K$ provided with a $K$-rational point. Let $\Ctilde\to C$ be the pullback of 
the multiplication-by-$2$ map on $\Jac(C)$ via the embedding $C\to\Jac(C)$ that
sends $P$ to the identity. We say that two curves $C$ and $C'$ over $K$ are
\emph{isogenous} if their Jacobians are isogenous over $K$, and we say that two 
pointed curves $(C,P)$ and $(C',P')$ are \emph{doubly isogenous} if $C$ and $C'$
are isogenous and $\Ctilde$ and $\Ctilde'$ are isogenous. (Similar definitions
can be made using pullbacks of other isogenies of the Jacobian; following 
Sutherland and Voloch, we focus here on the multiplication-by-$2$ map because it
is perhaps the simplest nontrivial choice.)

As we noted above, if $K$ is a finite field then $C$ and $C'$ are isogenous if 
and only if they have the same zeta function over $K$. One can use arithmetic 
statistics to develop heuristics for the number of pairs of pointed curves 
$(C,P)$ and $(C',P')$ that are isogenous or doubly isogenous. The hope is to
identify families of curves such that no two members are expected to be doubly
isogenous, and then prove that expectation and obtain a deterministic factoring
algorithm. As a first step, we would like gather data to test whether our
heuristics are reasonable, and to that end we study families for which the
heuristics suggest that there \emph{do} exist doubly isogenous pairs.

In as-yet-unpublished work, Howe, Sutherland, and Voloch studied genus-$2$ 
curves having an automorphism of order $3$; the full automorphism group of these
curves contains the dihedral group of order $12$. They gave a heuristic for the
number of such curves over finite fields of characteristic not $2$ or $3$ that 
are doubly isogenous. Their data showed that the number of such pairs was larger
than expected. This over-abundance was explained by the existence of a pair of
doubly isogenous pointed curves over the number field $\QQ(\sqrt{29})$; for 
every prime $\frakp$ of this number field, the reductions of these curves modulo
$\frakp$ gives a pair of doubly isogenous pointed curves having an automorphism
of order~$3$.

We explore another instance of this problem, working over a field $K$ of
characteristic not $2$ containing a primitive $4$th root of unity. We consider
curves of genus $2$ with an automorphism $\rho$ of order~$4$. The automorphism
groups of these curves contain the dihedral group of order~$8$. We study their
elementary abelian $2$-group covers, in some cases restricting to the situation 
where the Weierstrass points of the genus-$2$ curves are $K$-rational. An
imprecise summary of our main results is that, taking $K = \Fq$ for a prime
$q\equiv 1\bmod 4$:
\begin{enumerate}
\item The number of pairs of such curves over $\Fq$ whose Jacobians are
      isogenous over $\Fq$ grows as expected; 
      see Theorem~\ref{T:justthecurves} and Table~\ref{Table:JustCurves}.
\item The number of pairs of such curves that are ``$[1-\rho^*]$-isogenous'' 
      over $\Fq$ grows as expected;
      see Example~\ref{ex:heurzeta} and Table~\ref{Table:TwoTables}(B).
      (The terminology is explained in Example~\ref{ex:heurzeta}, but roughly
      speaking, the definition of being $[1-\rho^*]$-isogenous is the same as 
      that of being doubly isogenous, except the multiplication-by-$2$ 
      map on $\Jac(C)$ is replaced by a degree-$4$ endomorphism 
      of~$\Jac(C)$.)
\item The number of pairs of such curves that are doubly isogenous over $\Fq$
      is larger than expected;
      see Lemma~\ref{lem:ev1} and Table~\ref{Table:TwoTables}(A).  
      We explain this discrepancy in Section~\ref{sec:unexpected} by finding
      unexpected relationships between the Prym varieties of certain covers.
\end{enumerate}

We remark that this family of curves is Moonen's fourth special family
\cite{Moonen2010}. It would be interesting to study isogenies between curves in
the other special families of Moonen.

\subsection*{Conventions}
A \emph{curve} over a field $K$ is a smooth projective geometrically connected
$1$-dimensional variety over $K$. If $C$ is a curve over a field $K$, then 
$\Aut(C)$ is the group of $K$-rational automorphisms of $C$; if $L$ is an 
extension field of $K$, then $\Aut_L(C)$ is the group of $L$-rational 
automorphisms of $C$.

\section{The family of genus-\texorpdfstring{$2$}{2} curves with 
         \texorpdfstring{$D_4$}{D4}-action}
\label{sec:D4curves}

Let $K$ be a field of characteristic not~$2$. In this section, we find an
equation that describes every genus-$2$ curve over $K$ whose automorphism group
contains the dihedral group $D_4$ of order~$8$. We also show that the Jacobian
of a genus-$2$ curve with $D_4$ contained in its automorphism group is isogenous
to the square of an elliptic curve.

We fix a presentation of the dihedral group of order $8$:
\[D_4 =  \langle a, b \, | \, a^2 = b^2 = (ab)^4 = 1  \rangle.\]
We let $\xi$ denote the automorphism of $D_4$ that interchanges $a$ and $b$.

\begin{notation}\label{N:ZZ}
A \emph{curve with $D_4$-action} is a curve $Z$ together with an embedding
$\epsilon\colon D_4 \hookrightarrow \Aut(Z)$. We say that two curves with 
$D_4$-action $(Z,\epsilon)$ and $(Z',\epsilon')$ are \emph{isomorphic} if there
is a $K$-rational isomorphism $\varphi\colon Z \to Z'$ such that the following
diagram commutes:
\[
\xymatrix{
& D_4\ar[ld]_\epsilon\ar[rd]^{\epsilon'} \\
\Aut(Z)\ar[rr]_{\delta\mapsto \varphi\circ\delta\circ\varphi^{-1} } && \Aut(Z')
}
\]
Let $\CZ$ denote the set of $K$-isomorphism classes of genus-$2$ curves with
$D_4$-action over $K$. Using Igusa's classification~\cite[\S8]{Igusa1960} of the
automorphism groups of genus-$2$ curves, we check that if $(Z,\epsilon)$ is a 
genus-$2$ curve with $D_4$-action, then $\epsilon((ab)^2)$ is the hyperelliptic
involution.
\end{notation}

\begin{remark}
\label{rem:conjugation}
If $(Z,\epsilon)$ is a curve with $D_4$-action and $\alpha$ is an automorphism
of $Z$, we let $\epsilon^\alpha$ denote the inclusion
$D_4\hookrightarrow\Aut(Z)$ that sends $x$ to $\alpha \epsilon(x) \alpha^{-1}$.
Note that $\alpha\colon Z\to Z$ then gives a morphism of pairs
$(Z,\epsilon)\to(Z,\epsilon^\alpha)$, which shows that conjugating $\epsilon$ by
an automorphism of $Z$ does not change the isomorphism class of the pair
$(Z,\epsilon)$.
\end{remark}

\subsection{A family of genus-\texorpdfstring{$2$}{2} curves with 
            \texorpdfstring{$D_4$}{D4}-action}
\label{ssec:automorphisms}

Let $c$ and $s$ be elements of $K$ with $c\ne 0$ and $s\ne\pm 2$, and let $Z$ be
the genus-$2$ curve
\begin{equation} \label{EZstwist}
Z\colon\quad y^2 = c (x^2 + 1) (x^4 + s x^2 + 1).
\end{equation}
The hyperelliptic involution $\kappa$ of $Z$ is given by $(x,y)\mapsto(x,-y)$.
The curve $Z$ also has other $K$-rational involutions, including 
\[
    \sigma \colon (x,y)\mapsto (-x,y)      \text{\qquad and \qquad}
    \tau   \colon (x,y)\mapsto (1/x,y/x^3).
\]
Let $\rho = \sigma\tau$, so that $\rho$ takes $(x,y)$ to $(-1/x,y/x^3)$. We note
that $\rho^2 = \kappa$. The group $G$ generated by $\sigma$ and $\tau$ is
isomorphic to $D_4$. More precisely, we specify an inclusion
\begin{align}
\label{EZstwisteps}
\epsilon\colon D_4 &\hookrightarrow \Aut(Z)\\
\notag
a &\mapsto \sigma\\
\notag
b &\mapsto \tau.
\end{align}
Thus, Equations~\eqref{EZstwist} and~\eqref{EZstwisteps} give us a family of
genus-$2$ curves with $D_4$-action.

\begin{remark} \label{remark:igusa}
When $s \in \{-6, -1, 14\}$, the curve \eqref{EZstwist} has geometric
automorphism group strictly larger than $D_4$; using Igusa's
classification~\cite[\S8]{Igusa1960}, we can show that all other values of $s$
give curves with geometric automorphism group isomorphic to $D_4$.
\end{remark}

\subsection{Classifying genus-\texorpdfstring{$2$}{2} curves with 
            \texorpdfstring{$D_4$}{D4}-action up to isomorphism} 
\label{sec:classifying}

\begin{lemma}
\label{L:moduli}
Let $(Z',\epsilon')$ be a genus-$2$ curve over $K$ with $D_4$-action. Then there
are values $c,s \in K$ such that the curve with $D_4$-action $(Z,\epsilon)$
given by~\eqref{EZstwist} and~\eqref{EZstwisteps} is isomorphic to
$(Z',\epsilon')$. The value of $s$ is unique, and the value of $c$ is unique up
to multiplication by elements of $K^{\times 2}$.
\end{lemma}

\begin{proof}
Let $\alpha = \epsilon'(a)$, let $\beta = \epsilon'(b)$, and let $\iota$ be the
hyperelliptic involution on $Z'$. The quotient of $Z'$ by $\alpha$ has genus~$1$
because $\alpha\neq\iota$. The Riemann--Hurwitz formula shows that $\alpha$ has
two geometric fixed points. If $P$ is one of these fixed points, then 
$\iota P = \iota\alpha P = \alpha \iota P$, so $\iota P$ is fixed by $\alpha$ as
well. 

We claim that $P \neq \iota P$. To see this, consider the $V_4$-subgroup 
$H = \langle \alpha, \iota \rangle$. The stabilizer of any point under $H$ is
the decomposition group of the corresponding place in the geometric cover
$Z' \to Z'/H$. This decomposition group is cyclic since the extension is tamely
ramified. Hence, no fixed point of $\alpha$ is fixed by $\iota$. 

Consider the quotient $\PP = Z'/\langle\iota\rangle$. The two fixed points of
$\alpha$ are $P$ and $\iota P$. These two points map to the same point $Q$ in 
$\PP$ and $Q$ must be $K$-rational. Since $\alpha$ and $\beta$ both commute with
$\iota$, they descend to automorphisms $\alphabar$ and $\betabar$ of $\PP$. The 
point $Q$ is one of the fixed points of the involution $\alphabar$, so both
fixed points of $\alphabar$ must be $K$-rational. By choosing the coordinate $x$
on $\PP$ appropriately, we may assume that the fixed points of $\alphabar$ are
$x=0$ and $x=\infty$. This means that an equation for $Z'$ is 
\[y^2 = a_6 x^6 +  a_4 x^4 + a_2 x^2  + a_0,\]
for some constants $a_0, a_2, a_4, a_6 \in K$, and in this model $\alpha$ sends
$(x,y)$ to $(-x,y)$.

Since $(\alpha\beta)^2 = \iota$ and $\iota$ induces the trivial automorphism on
$\PP$, we see that $\alphabar\betabar$ is an involution of $\PP$, implying that
$\alphabar$ and $\betabar$ commute. This means that $\betabar$ must be a linear
fractional transformation of the form $x\mapsto d/x$ for some $d \in K^\times$. 
Since the fixed points of $\betabar$ are also $K$-rational, $d$ is a square in 
$K^\times$. By scaling $x$ by $\sqrt{d}$, we may assume that $d = 1$. This 
implies that $a_6=a_0$ and $a_4=a_2$, so that an equation for $Z'$ is
\[y^2 = a_0 x^6 +  a_2 x^4 + a_2 x^2  + a_0,\]
and so that $\beta$ sends $(x,y)$ to $(1/x,y/x^3)$. Let $c = a_0$ and
$s = a_2/a_0$, and let $(Z,\epsilon)$ be the genus-$2$ curve with $D_4$-action
given by~\eqref{EZstwist} and~\eqref{EZstwisteps}. Our model for $Z'$ gives us
an isomorphism $(Z',\epsilon')\to(Z,\epsilon)$.

Demanding that the fixed points of $\alphabar$ be $0$ and $\infty$ and that the
fixed points of $\betabar$ be $1$ and $-1$ completely specifies the standard
hyperelliptic model for $Z'$, up to scaling $y$ by a constant. These scalings
modify $c$ by multiplication by a square in $K^\times$. This proves the final
statement of the lemma.
\end{proof}

\begin{lemma}
\label{L:moduli2}
The two curves
\[
  Z\colon y^2 = c(x^2 + 1)(x^4 + s x^2 + 1) \text{\quad and }\ 
  Z'\colon y^2 = c'(x^2 + 1)(x^4 + s' x^2 + 1)
\]
are isomorphic to one another if and only if  either $c' = c$ 
\textup(in $K^\times/K^{\times 2}$\textup) and $s' = s$, or $c' = 2 c (s + 2)$
\textup(in $K^\times/K^{\times 2}$\textup) and $(s'+2)(s+2) = 16$.
\end{lemma}

Note that the lemma says that every curve of the form given by
Equation~\eqref{EZstwist} has exactly one other model of the same form, unless
$s = -6$ and $-2$ is a square, in which case the lemma claims that the given
model is unique.

\begin{proof}[Proof of Lemma~\textup{\ref{L:moduli2}}]
If either of the given relations among $c,c'$ and $s,s'$ hold, it is easy to
check that the two curves are isomorphic to each other. The isomorphism in the 
second case is given by $(x,y)\mapsto \big( (x+1)/(x-1), y/(x-1)^3\big)$.

On the other hand, suppose we have a curve $Z$ as in the lemma. We would like to
see how many other models it has that are also of the form given by
Equation~\eqref{EZstwist}. Notation~\ref{N:ZZ}, Remark~\ref{rem:conjugation}, 
and Lemma~\ref{L:moduli} show that these models correspond to the embeddings of 
$D_4$ into $\Aut(Z)$, up to conjugation by $\Aut(Z)$, so we just need to count
the number of such embeddings up to conjugacy.

If $s\not\in\{-6,-1,14\}$ then $\Aut(Z)\cong D_4$ by Remark~\ref{remark:igusa}. 
The outer automorphism group of $D_4$ has two elements, so there are two
embeddings of $D_4$ into $\Aut(Z)$ up to conjugation and hence two models of the
form~\eqref{EZstwist}. These are accounted for by the two models in the lemma.

If $s\in\{-1,14\}$, then by computing Igusa invariants and 
consulting~\cite[\S8]{Igusa1960} we find that~$\Aut_{\Kbar}(Z)$ is a certain
group of order $24$, so $\Aut(Z)$ is a subgroup of this group that
contains~$D_4$. By enumeration, we find that for each such subgroup $G$ there
are two conjugacy classes of embedding $D_4\hookrightarrow G$. Once again, these
are accounted for by the two models in the lemma.

When $s=-6$, we find from Igusa that $\Aut_{\Kbar}(Z)$ is either a certain group
$G_{48}$ of order $48$ (if $K$ has characteristic not $5$) or a certain group
$G_{240}$ of order $240$ (if $K$ has characteristic~$5$). Both of these groups
contain a unique subgroup $G_{16}$ of order $16$. For every subgroup $G$ of 
$G_{48}$ or $G_{240}$ that contains $D_4$, we find that the number of conjugacy
classes of embeddings $D_4\hookrightarrow G$ is equal to $2$ if $G$ does not
contain $G_{16}$, and is equal to $1$ if $G$ does contain $G_{16}$.

In terms of the model $y^2 = c (x^2 + 1)(x^4 - 6x^2 + 1)$ for $Z$, the group
$G_{16}$ is generated by the involutions $\sigma$ and $\tau$ together with the
automorphism $\upsilon$ of order $8$ given by
$(x,y)\mapsto ((x-1)/(x+1), 2\sqrt{-2}\, y/(x+1)^3)$. We see that $G_{16}$ is
contained in $\Aut(Z)$ if and only if $-2$ is a square in $K$. Combined with the
results of the preceding paragraph, we find two models for $Z$ when $-2$ is not
a square, and one otherwise.
\end{proof}

\subsection{An invariant of the curve \texorpdfstring{$Z$}{Zs}}
\label{sec:invariant}

Lemma~\ref{L:moduli2} shows that two curves $Z$ and $Z'$ of the form given by 
Equation~\eqref{EZstwist} are geometrically isomorphic to one another if and 
only if either $s'=s$ or $s' = (-2s + 12)/(s + 2)$. The function 
\begin{equation} \label{eq:invariant}
    I(s) \colonequals -\frac{(s-2)^2}{4(s+2)} = 1 - \frac{s+s'}{4}
\end{equation}
is stable under the involution $s\leftrightarrow s'$ and is rational of
degree~$2$, so it gives a geometric invariant for the curve $Z$.

\subsection{Structure of the Jacobian of the curve \texorpdfstring{$Z$}{Zs}}
\label{ssec:structureofJac}

In this section, we consider the quotients of $Z$ by the non-central involutions
of $D_4$. 

Let $E$ be the elliptic curve defined by
\begin{equation} 
\label{EdefE} E\colon\quad y^2 = c(x + 1)(x^2 + sx + 1).
\end{equation}

\begin{lemma}
\label{lem:quot1}
The quotient of $Z$ by each of the involutions $(x,y)\mapsto(-x,\pm y)$ is
isomorphic to $E$, and $\Jac(Z)$ is isogenous to $E^2$.
\end{lemma}

\begin{proof}
The quotient of $Z$ by the involution $(x,y)\mapsto(-x,y)$ is clearly $E$.

To find the quotient by $(x,y)\mapsto(-x,-y)$, it helps to rewrite the equation
for $Z$ as 
\[ x^2 y^2 = c x^2 (x^2 + 1) (x^4 + sx^2 + 1).\]
Since $xy$ and $x^2$ are both fixed by the involution, the quotient is given by
the equation
\[ y^2 = cx(x+1)(x^2 + sx + 1).\]
If we replace $(x,y)$ with $(1/x,y/x^2)$, we obtain \eqref{EdefE} and hence the
second quotient is isomorphic to $E$.

These two involutions generate a subgroup of $\Aut(Z)$ isomorphic to the Klein
four-group. The product of these involutions is the hyperelliptic involution
$\kappa$; the quotient of $Z$ by $\kappa$ is the projective line.
By~\cite[Theorem C]{KaniRosen1989}, $\Jac(Z)$ is isogenous to the product of the
Jacobians of the three quotients, thus $\Jac(Z) \sim \Jac(E)^2 \cong E^2$.
\end{proof}

Let $s' = (-2s+12)/(s+2)$ and $c' = 2c(s+2)$, and let $E'$ be the elliptic curve
\[
E'\colon\quad y^2 = c' (x + 1) (x^2 + s'x + 1).
\]
Note that there is a $2$-isogeny $E\to E'$ given by
\[
(x,y)\mapsto 
\Bigg( \frac{1}{s+2} \frac{(2 x + s)(x - 1)}{(x + 1)},
\frac{4}{s+2} \frac{(x^2 + 2 x + s - 1)}{(x+1)^2}\ y \Bigg),
\]
whose kernel contains the $2$-torsion point $P = (-1,0)$ of $E$. The kernel of
the dual isogeny $E'\to E$ contains the $2$-torsion point $P'=(-1,0)$ of $E'$.

\begin{lemma}
\label{lem:quot2}
The quotient of $Z$ by each of the involutions $(x,y)\mapsto(1/x,\pm y/x^3)$ is 
isomorphic to~$E'$.
\end{lemma}

\begin{proof}
Replacing $x$ with $(x+1)/(x-1)$ and $y$ with $y/(x-1)^3$ in the equation
for~$Z$, we find that $Z$ can also be written as
\[
y^2 = c' (x^2 + 1) (x^4 + s'x^2 + 1).
\]
The two involutions $(x,y)\mapsto(1/x,\pm y/x^3)$ in the original model become
the involutions $(x,y)\mapsto(-x,\mp y)$ in the new model. The result follows 
from Lemma~\ref{lem:quot1}.
\end{proof}

Lemma~\ref{lem:quot1} says that there is an isogeny $E^2\to\Jac(Z)$, but we can
be much more precise.

\begin{proposition}
\label{prop:JacobianDecomposition}
Let $E$ be as above, let $P = (-1,0)\in E[2]$, and let $Q$ and $R$ be the other
two geometric points of order $2$ on $E$. Let $\psi\colon E[2]\to E[2]$ be the
automorphism that fixes $P$ and swaps $Q$ and $R$. Then there is an isogeny
$\varphi\colon E\times E\to \Jac(Z)$ whose kernel is the graph of $\psi$, and
the pullback via $\varphi$ of the principal polarization on $\Jac(Z)$ is twice
the product polarization on $E\times E$.
\end{proposition}

\begin{proof}
Because there is a degree-$2$ map $Z\to E$, the general theory set out in
\cite[\S2]{Kani1997} shows that there is an elliptic curve $F$, an isomorphism
$\psi\colon E[2]\to F[2]$, and an isogeny $E\times F\to \Jac(Z)$ satisfying the
conclusion of the proposition. The explicit construction carried out in
\cite[\S3]{HoweLeprevostPoonen2000} shows that $F\cong E$ and that $\psi$ is the
isomorphism specified in the statement.
\end{proof}

In fact, almost every pair $(E,P)$ consisting of an elliptic curve $E$ over $K$
and a $K$-rational $2$-torsion point arises in this way.

\begin{proposition}
\label{prop:ZfromEP}
Let $E$ be an elliptic curve over $K$ with a rational point $P$ of order $2$.
Then there is a genus-$2$ curve $Z$ over $K$ with $D_4$-action that gives rise
to this $(E,P)$ as above if and only if $E$ does \emph{not} have a geometric
automorphism $\alpha\neq\pm 1$ that fixes~$P$.
\end{proposition}

\begin{proof}
Given an $E$ and a $P$ as in the statement of the proposition, we may choose a
model $y^2 = x(x^2 + ax + b)$ for $E$ so that $P$ is the point $(0,0)$. Let
$\psi\colon E[2]\to E[2]$ be the automorphism that fixes $P$ and swaps the other
two points of order $2$. If there is no geometric automorphism of $E$ that
restricts to $\psi$ on $E[2]$, then the construction 
of~\cite[Proposition~4, p.~324]{HoweLeprevostPoonen2000} produces a genus-$2$
curve $Z$ of the form~\eqref{EZstwist}, and we check that the $E$ and $P$
produced by this curve as above are the $E$ and $P$ we started with.

If there \emph{is} a geometric automorphism $\alpha$ of $E$ that restricts to
$\psi$, then the geometric isogeny $\varphi\colon E\times E \to E\times E$ that
takes $(U,V)$ to $(U + \alpha^{-1}(V), V-\alpha(U))$ has kernel equal to the
graph of $\psi$, and the pullback via $\varphi$ of the product polarization is
twice the product polarization. If there were a curve $Z$ that gave rise to
$(E,P)$, then by Proposition~\ref{prop:JacobianDecomposition} the polarized
Jacobian of $Z$ would be geometrically isomorphic to $E\times E$ with the 
product polarization, which is impossible. To complete the proof, we just need
to observe that over fields of characteristic not~$2$, every automorphism 
$\alpha\ne \pm 1$ of an elliptic curve that fixes one point of order $2$ 
necessarily swaps the other two.
\end{proof}

\subsection{Related families of genus-\texorpdfstring{$2$}{2} curves
            with \texorpdfstring{$D_4$}{D4}-action}

If $K$ is algebraically closed, Cardona and Quer
\cite[Proposition 2.1]{CardonaQuer2007} show that every genus-$2$ curve $Y$ with
$\Aut(Y)\cong D_4$ is a member of the family
\[
Y_v \colon \quad y^2 = x^5 + x^3 + vx,
\]
where $v\in K\setminus\{0, 1/4, 9/100\}$. 


The advantage of this family is that every geometric isomorphism class of a
curve with automorphism group $D_4$ corresponds to exactly one value of $v$, as
opposed to the family $Z$ in \eqref{EZstwist}. The disadvantage is that the
automorphisms of this curve are not necessarily defined over the field generated
by the parameter~$v$. 

Moonen \cite{Moonen2010} studied cyclic covers of ${\mathbb P}^1$ given by
monodromy data. One of the twenty families of curves that appear in his work is
the cyclic degree-$4$ cover of $\PP$ given by
\[
X_T \colon\quad  z^4 = x(x-1)^2(x-T)^2.
\]
The curve $X_T$ has  genus~$2$ and $\Aut_{\Kbar}(X_T) \cong D_4$ for a generic
choice of $T$. This model makes the order-$4$ automorphism very apparent, but 
the hyperelliptic structure is not as clearly visible.


\section{The 2-torsion and unramified elementary abelian 2-covers}

In this section, we study unramified elementary abelian 2-covers of the curves
$Z$ defined by \eqref{EZstwist}. Throughout this section, we assume all
Weierstrass points of $Z$ are defined over $K$.  Let $\zeta$ be a primitive
fourth root of unity in $K$. Then $Z$ can be given by
\begin{gather}
\label{EQ:D4withWeierstrass}
   Z\colon\quad y^2 = c(x-\zeta)(x+\zeta)(x-t)(x+t)(x-1/t)(x+1/t), 
\text{\quad where } \\
\label{Estformula}
   s = -(t^4+1)/t^2.
\end{gather}

\subsection{Unramified elementary abelian 2-covers} \label{ss:curvescovers}

Let $P \in Z(K)$ be the Weierstrass point $(\zeta,0)$. Since there is a 
$K$-rational automorphism of $Z$ taking $P$ to $(-\zeta, 0)$, the choice of
$\zeta$ does not affect the $K$-isomorphism class of the $2$-covers we 
construct. Note that $P$ and $(-\zeta, 0)$ are distinguished from the other
Weierstrass points of $Z$ by the fact that they form an orbit of size $2$ under
the action of $D_4$; the other Weierstrass points form an orbit of size~$4$.

\begin{definition}
Let $\iota_{P}\colon Z \hookrightarrow \Jac(Z)$ be the Abel-Jacobi embedding 
that sends $Q \in Z(\Kbar)$ to the divisor class $[Q - P]$. Let 
$\pi^{\mathbf{0}}\colon \Ztilde \to Z$ be the pullback of the 
multiplication-by-$2$ map on $\Jac(Z)$ by 
$\iota_P\colon Z \hookrightarrow \Jac(Z)$.
\end{definition}

Our assumption that the Weierstrass points of $Z$ are $K$-rational implies that
the cover $\Jac(Z)\to\Jac(Z)$ given by the multiplication-by-$2$ map is Galois,
with Galois group isomorphic to $\Jac(Z)[2]\cong (\ZZ / 2 \ZZ)^{4}$; the group 
$\Jac(Z)[2]$ acts on the cover by translation. Since $\pi^{\mathbf{0}}$ is 
defined as a pullback of this cover, $\pi^{\mathbf{0}}$ is also Galois, with 
Galois group canonically isomorphic to $\Jac(Z)[2]$. In fact, geometric class
field theory shows that we can recognize $\pi^{\mathbf{0}}$ as the maximal
unramified abelian extension of $Z$ with Galois group of exponent $2$ in which
the base point $P = (\zeta,0)$ splits completely.

\begin{definition} \label{DpiH}
For a subgroup $H$ of $\Jac(Z)[2]$, let $\Ztilde^H$ be the quotient of $\Ztilde$
by $H$. Let $\pi^H\colon \Ztilde^H \to Z$ be the quotient cover.
\end{definition}

For example, $\Ztilde^{\mathbf{0}} = \Ztilde$ and $\Ztilde^{\Jac(Z)[2]} = Z$.
More generally, the degree of $\pi^H$ equals the index of $H$ in $\Jac(Z)[2]$.
Since $\Jac(Z)[2]$ is abelian, $\pi^H$ is Galois with Galois group isomorphic to
$\Jac(Z)[2] /H$. Furthermore, the genus of $\Ztilde^{H}$ equals
$[\Jac(Z)[2]:H]+1$ by the Riemann--Hurwitz formula. 

\begin{remark}
\label{rmk:DifferentBasepoint}
If we pick a different basepoint $P'$ instead of $P$ and keep track of the
basepoint dependence by labeling the 2-covers as $\Ztilde_{P}$ and
$\Ztilde_{P'}$, and if we let $Q\in\Jac(Z)(\Kbar)$ be a point with 
$2Q = P - P'$, then translation by $Q$ on $\Jac(Z)(\Kbar)$ yields a geometric
isomorphism from $\Ztilde_{P}$ to $\Ztilde_{P'}$. We will prove in the following
paragraph that there exists an elementary abelian 2-extension $L$ of $K$ of
degree at most $2^{4}$ such that $Q\in \Jac(Z)(L)$, so this translation
isomorphism will be defined over $L$. In particular, when $K$ is a finite field,
then $L$ is at worst a quadratic extension of $K$, and the curve $\Ztilde_{P'}$
is a (possibly trivial) quadratic twist of $\Ztilde_{P}$. From now on, we fix
the base point to be $P=(\zeta,0)$ for all $\pi^H$.

The following argument was provided by Bjorn Poonen. The obstruction to dividing
a point of $\Jac(Z)(K)$ by 2 lies in $H^{1}(K, \Jac(Z)[2])$. Since all the
Weierstrass points are defined over $K$, we know that
$\Jac(Z)[2] \cong \mu_{2}^{4}$ as a Galois module. Hence, the obstruction to
dividing $P' - P$ by $2$ lies in  
$H^{1}(K, \Jac(Z)[2]) \cong (H^1(K, \mu_{2}))^{4} 
   = (K^\times / K^{\times 2})^{4}$,
so there exists an elementary abelian 2-extension $L / K$ of degree at most
$2^{4}$ such that the image of this class in 
$H^1(L, \Jac(Z)[2]) = (L^\times / L^{\times 2})^{4}$ is trivial. When $K$ is a
finite field of characteristic not $2$, over its unique elementary abelian
$2$-extension $L$, this obstruction class vanishes. Thus $\Ztilde_{P'}$ is a
quadratic twist of $\Ztilde_{P}$.
\end{remark}

\subsection{Decomposition of the Jacobian}

In this section, we determine the isogeny decomposition of $\Jac(Z^H)$ over $K$.

\begin{definition}
Given a cover $\pi\colon V \to Z$, let $\Prym^\pi$ denote the Prym variety of
$\pi$, that is, the identity component of the kernel of the induced norm
homomorphism $\Jac(V) \to \Jac(Z)$. There is a $K$-isogeny
$\Jac(V) \sim \Jac(Z) \times \Prym^\pi.$
\end{definition}

\begin{definition}
\label{def:PrymH}
For a subgroup $H$ of $\Jac(Z)[2]$, we set $\Prym^H \colonequals \Prym^{\pi^H}$,
where $\pi^H$ is the cover defined in Definition~\ref{DpiH}.
\end{definition}

\begin{proposition} \label{prop:decomposecover}
Let $Z$ be a genus-$2$ curve with $D_4$-action whose Weierstrass points are
defined over $K$. For every $H \subseteq \Jac(Z)[2]$, there is an isogeny
\begin{equation}
\label{JacDecomposePrym}
\Jac(\Ztilde^H) \sim E^2 \times 
\prod_{\stackrel{H \subseteq H' \subseteq \Jac(Z)[2]}
                {\scriptscriptstyle [\Jac(Z)[2] : H'] =2}} \Prym^{H'}.
\end{equation}
\end{proposition}

\begin{proof}
Let $G = \Jac(Z)[2]$ and $r = [G:H]$, and enumerate the index-$2$ subgroups of
$G$ containing $H$ by $H_1', \cdots, H_{r - 1}'$. We apply 
\cite[Theorem C]{KaniRosen1989} to the $\{H_i'\}$ together with $H$ and $G$.
More precisely, we define 
\[
H_{i} \colonequals \begin{cases}
                   H_{i}' &\text{for } i = 1, \ldots r - 1, \\
                   H &\text{for } i = r, \\
                   G &\text{for } i = r + 1,
                   \end{cases} 
\quad \text{and} \quad 
n_{i} \colonequals \begin{cases}
                   -1     &\text{for } i = 1,\ldots r - 1, \\
                    1     &\text{for } i = r, \\
                    r - 2 &\text{for } i = r + 1.
\end{cases}
\]
Let $g_{i j}$ be the genus of $Z^{H_{i} H_{j}}$. The group $H_{i} H_{j}$ must be
one of $H_{1}', \cdots, H_{r - 1}', H, G$.  We know from Riemann--Hurwitz that
the genus of $\Ztilde^{H}$ is $r + 1$, the genus of each $\Ztilde^{H_{i}'}$ 
is~$3$, and the genus of $\Ztilde^{G}$ is $2$. Using this information and some
casework, we see that
\[
g_{i j} = \begin{cases}
         2     &\text{if } i \neq j \text{ and } i, j \le r - 1, \\
         3     &\text{if } i = j \le r - 1, \\
         3     &\text{if } i \le r - 1\text{ and } j = r\text{, or }
                           i = r\text{ and }j \le r - 1, \\
         r + 1 &\text{if } i = j = r, \\
         2     &\text{if } i\text{ or }j \text{ is } r + 1. \\
\end{cases}
\]
We check the conditions to apply \cite[Theorem C]{KaniRosen1989}: first, 
$H_{i} H_{j} = H_{j} H_{i}$ is satisfied because $G$ is abelian; and second, 
$\sum_i n_{i} g_{i j} = 0$ for all $j \in \{1,\ldots,r+1\}$ by our computations
above. Therefore the conclusion of \cite[Theorem C]{KaniRosen1989} holds, 
namely, there exists a $K$-isogeny
\[
\prod_{n_i > 0} \Jac(\Ztilde^{H_{i}})^{n_{i}} 
     \sim \prod_{n_{j} < 0} \Jac(\Ztilde^{H_{j}})^{-n_{j}},
\]
which for us becomes
\[
\Jac(\Ztilde^{H}) \times (\Jac(Z))^{r - 2} 
     \sim \prod_{\stackrel{H \subseteq H' \subseteq \Jac(Z)[2]}
                { \scriptscriptstyle [\Jac(Z)[2] : H'] =2}} \Jac(\Ztilde^{H'}).
\]
Now we substitute $\Jac(\Ztilde^{H'}) \sim \Jac(Z) \times \Prym^{H'}$, cancel
$(\Jac(Z))^{r - 2}$ from both sides, and substitute $\Jac(Z) \sim E^2$ (from 
Lemma \ref{lem:quot1}) to finish.
\end{proof}

\subsection{The Weil pairing}
\label{ssec:WeilPairing}

\begin{definition} \label{Dweier}
Let $R = \{\zeta,-\zeta,t,-t,1/t,-1/t\}$. For $r \in R$, let $W_{r}$ denote the
Weierstrass point $(r,0)$ of $Z$.
\end{definition}

\begin{lemma}
If $D \in \Jac(Z)[2]$, then there exist $u, v\in R$ such that $D = [W_u - W_v]$. 
\end{lemma}

\begin{proof}
We know $\Jac(Z)[2]$ is generated by $[W_{r} - W_{\zeta}]$ for
$r \in R \setminus \{\zeta\}$ with the single relation 
$\sum_{r} [W_{r} - W_{\zeta}]=0$. The conclusion follows from a straightforward
computation.
\end{proof}

\begin{definition}
Let $e_{2}(\cdot, \cdot)$ denote the Weil pairing on $\Jac(Z)[2]$, which takes
values in $\{\pm 1\}\subset K^\times$. For every subgroup $H$ of $\Jac(Z)[2]$, 
define $H^\perp$ by
\[
H^\perp \colonequals 
        \{ S \in \Jac(Z)[2]\ :\ e_2(Q, S) = 1 \text{ for all } Q \in H \}.
\]
\end{definition}

For later use, we give an explicit description of the Weil pairing on
$2$-torsion points.

\begin{lemma} \label{lem:explicitweil}
For nonzero elements $[W_{u_1} - W_{v_1}]$ and $[W_{u_2} - W_{v_2}]$ of
$\Jac(Z)[2](K)$, we have $e_2( [W_{u_1} - W_{v_1}], [W_{u_2} - W_{v_2}]  ) = -1$
if and only if $\# (\{u_1,v_1\} \cap \{u_2,v_2\}) = 1$.
\end{lemma}

\begin{proof}
This is a direct calculation using a well-known formula for the Weil pairing
(see \cite[Theorem 1]{Howe1996}).  
\end{proof}

\begin{proposition}
\label{PropPrymJacobianElliptic}
Let $H'$ be an index-$2$ subgroup of $\Jac(Z)[2]$ and let $U = [W_u-W_v]$ be the
generator of $(H')^{\perp}$. Define $a\in K^\times/K^{\times 2}$ by
\[ a \colonequals \begin{cases}
  c (\zeta - u)(\zeta - v) & \text{if $\zeta\not\in\{u,v\}$}\\
  \prod_{r\in R\setminus\{u,v\}} (\zeta - r) & \text{if $\zeta\in\{u,v\}$}.
  \end{cases}
\]
Let $E'$ be the genus-$1$ curve given by the equation 
$y^{2} = a \prod_{r\in R\setminus\{u,v\}} (x - r).$ Then there is a $K$-isogeny
$\Prym^{H'} \sim \Jac(E').$
\end{proposition}

\begin{proof}
Let $f_0 = ac(x-u)(x-v)$ and $f_1 = a \prod_{r\in R\setminus\{u,v\}} (x - r)$,
and consider the $V_4$-diagram of function fields
\[
\xymatrix@=2ex{
&&K(x,\sqrt{f_0},\sqrt{f_1}) \ar@{-}[lld]\ar@{-}[rrd]\ar@{-}[d]&&\\
K(x,\sqrt{f_0})\ar@{-}[rrd] &&K(x,\sqrt{f_0 f_1})\ar@{-}[d] 
                            &&K(x,\sqrt{f_1})\ar@{-}[lld]\\
&&K(x)\rlap{.}&&
}
\]
If we let $C$ be the genus-$0$ curve $y^2 = f_0$, the diagram above gives us a
$V_4$-diagram of curves
\begin{equation}
\label{EQ:V4curves}
\xymatrix@=2ex{
&&Y \ar[lld]\ar[rrd]\ar[d]&&\\
C\ar[rrd] &&Z\ar[d] &&E'\ar[lld]\\
&&\PP\rlap{,}&&
}
\end{equation}
where $Y$ is the curve with function field $K(x,\sqrt{f_0},\sqrt{f_1})$. The 
value of $a$ was chosen so that the point $x=\zeta$ of $\PP$ splits in one of
the extensions $C\to\PP$ and $E'\to\PP$ and ramifies in the other, and it
follows that the Weierstrass point $P = (\zeta,0)$ of $Z$ splits in the
quadratic extension $Y\to Z$. Since $Y\to Z$ is a Galois extension with group of
exponent $2$ in which $P$ splits completely, it must be a subextension of
$\Ztilde\to Z$, which as we noted earlier is the maximal such extension. This
tells us that the element $f_0$ of $K(Z)^\times$ is a square in
$K(\Ztilde)^\times$.

In fact, we see that the map $U\mapsto f_0$ defines an injective homomorphism
$\gamma\colon\Jac(Z)[2]\to 
  (K(Z)^\times \cap K(\Ztilde)^{\times 2})/K(Z)^{\times 2}$.
If we let $G$ be the latter group, then Kummer theory says that there is a
perfect pairing 
\[
\Gal(\Ztilde/Z) \times G  \to \{\pm 1\}\subset K^\times.
\]
In particular, $\#G = \#\Gal(\Ztilde/Z) = \#\Jac(Z)[2]$, so the injective
homomorphism $\gamma$ is an isomorphism. This isomorphism, together with the
canonical isomorphism $\Jac(Z)[2]\cong\Gal(\Ztilde/Z)$, turns the Kummer
pairing into a perfect pairing
\[
\Jac(Z)[2] \times \Jac(Z)[2] \to \{\pm 1\}.
\]
As is observed in~\cite[\S2]{Howe1996}, this pairing is in fact the Weil
pairing; this can be seen by using the explicit formula for the natural pairing
of the $m$-torsion of an abelian variety with that of its 
dual~\cite[\S16]{Milne1986av} and the fact that the pullback of the Abel--Jacobi
map $Z\to\Jac(Z)$ is equal to $-\lambda^{-1}\colon \hat{\Jac(Z)} \to \Jac(Z)$, 
where $\lambda\colon \Jac(Z)\to\hat{\Jac(Z)}$ is the canonical polarization on
$\Jac(Z)$~\cite[Lemma~6.9 and Remark~6.10(c)]{Milne1986jv}.

From this we conclude the cover $Y\to Z$ is $\pi^{H'}$. Furthermore, we see
from Diagram~\eqref{EQ:V4curves} and~\cite[Theorem C]{KaniRosen1989} that
$\Jac(Y)\sim\Jac(Z)\times\Jac(E')$, so $\Prym^{H'}\sim\Jac(E').$
\end{proof}

\subsection{The \texorpdfstring{$D_4$}{D4}-action on the factors of
            \texorpdfstring{$\Jac(\Ztilde)$}{Jac(Z~)}}
\label{ss:orbits}

Applying Proposition \ref{PropPrymJacobianElliptic} to the fifteen index-$2$
subgroups of $\Jac(Z)[2]$ yields fifteen elliptic curves. Our notation for these
curves unfortunately depends on the value of $t\in K$ used in the defining
equation~\eqref{EQ:D4withWeierstrass} for $Z$; in the next subsection we will 
see what happens when we choose a different value of $t$ that defines a curve
isomorphic to $Z$.

\begin{definition}
\label{D:EllipticPryms}
Given a nonzero $U\in\Jac(Z)[2]$, let $E_U$ be the elliptic curve obtained by
applying Proposition~\ref{PropPrymJacobianElliptic} to the index-$2$ subgroup
$\langle U\rangle^\perp$. If $u$ and $v$ are the unique elements of $R$ such
that $U$ is equal to the divisor class $[W_u - W_v] = [W_v - W_u]$, we also
write $E_{\{u,v\}}$ for $E_U$.
\end{definition}

\begin{corollary}
\label{cor:jacstr}
There is a $K$-isogeny
\[
\Jac(\Ztilde) \sim E^2 \times \prod_{U} E_U,
\]
where the product is over nonzero $U\in \Jac(Z)[2].$
\end{corollary}
\begin{proof}
Combine Proposition \ref{prop:decomposecover} with $H = \mathbf{0}$ and
Proposition~\ref{PropPrymJacobianElliptic}.
\end{proof}

\begin{proposition}
\label{Porbit}
The set of nonzero elements of $\Jac(Z)[2]$ breaks up into six orbits under the
action of $D_4$, as listed in Table~\textup{\ref{tab:orbits}}. For each $U$ in
an orbit, the table presents a value $a\in K^\times$ as in
Proposition~\textup{\ref{PropPrymJacobianElliptic}}, and values of $\lambda$ and
$d$ such that $E_U$ is isomorphic to $y^2 = d x (x-1)(x-\lambda)$.
\end{proposition}
\begin{table}[ht]
{\begin{tabular}{ll@{\qquad}rrr}
\toprule
Orbit & Point &     &           &      \\
label & label & $a$ & $\lambda$ & $d$  \\
\midrule
\quad1     & $\{  \zeta, -\zeta \}$ & $1                   $ & $4       t^2 /(t^2 + 1)^2   $ & $1$                 \\[2ex]
\quad2A    & $\{      t,     -t \}$ & $c(t^2 + 1)          $ & $4 \zeta t   /(t + \zeta)^2$  & $c(t^2 + 1)$        \\
           & $\{    1/t,   -1/t \}$ & $c(t^2 + 1)          $ &                               &                     \\[2ex]
\quad2B    & $\{     -t,   -1/t \}$ & $\zeta c t (t^2 + 1) $ & $2(t^2 - 1)/(t + \zeta)^2$    & $c(t^2 + 1)$        \\
           & $\{      t,    1/t \}$ & $\zeta c t (t^2 + 1) $ &                               &                     \\[2ex]
\quad2C    & $\{      t,   -1/t \}$ & $\zeta c t           $ & $-1$                          & $c(t^2 + 1)$        \\
           & $\{     -t,    1/t \}$ & $\zeta c t           $ &                               &                     \\[2ex]
\quad4A    & $\{  \zeta,    1/t \}$ & $\zeta   t (t+\zeta) $ & $-2\zeta t/(t - \zeta)^2$     & $\zeta$             \\
           & $\{  \zeta,     -t \}$ & $          (t+\zeta) $ &                               &                     \\[1ex]
           & $\{ -\zeta,   -1/t \}$ & $\zeta c t (t-\zeta) $ & $-2\zeta t/(t - \zeta)^2$     & $\zeta c (t^2 + 1)$ \\
           & $\{ -\zeta,      t \}$ & $      c   (t-\zeta) $ &                               &                     \\[2ex]
\quad4B    & $\{  \zeta,      t \}$ & $          (t-\zeta) $ & $ 2\zeta t/(t + \zeta)^2$     & $\zeta$             \\
           & $\{  \zeta,   -1/t \}$ & $\zeta   t (t-\zeta) $ &                               &                     \\[1ex]
           & $\{ -\zeta,     -t \}$ & $      c   (t+\zeta) $ & $ 2\zeta t/(t + \zeta)^2$     & $\zeta c (t^2 + 1)$ \\
           & $\{ -\zeta,    1/t \}$ & $\zeta c t (t+\zeta) $ &                               &                     \\
\bottomrule    
\end{tabular}
}
\vskip 2ex
\caption{The fifteen elliptic curves $E_U$ for nonzero $U\in \Jac(Z)[2]$,
labeled as in Definition~\ref{D:EllipticPryms}, grouped in their orbits under
the action of $D_4$. The value of $a$ is as in
Proposition~\ref{PropPrymJacobianElliptic}, and the values of $\lambda$ and $d$
are such that $E_U$ is also isomorphic to $y^2 = d x(x-1)(x-\lambda)$. Recall
that each $E_U$ can be recovered as the Prym variety of the double cover of $Z$ 
associated, as in Proposition~\ref{PropPrymJacobianElliptic}, to the subgroup 
$H'\subseteq \Jac(Z)[2]$ that pairs trivially with $U$ under the Weil pairing.
}
\label{tab:orbits}
\end{table}

\begin{proof}
The generators $\sigma$ and $\tau$ of the $D_{4}$ subgroup of $\Aut(Z)$ act on 
the curve labels via
\[
\sigma(\{u,v\}) = \bigl\{ -u, -v \bigr\} \text{\quad and \quad} 
  \tau(\{u,v\}) = \bigl\{ 1/u, 1/v \bigr\},
\]
so the grouping into orbits is clear. The value of $a$  is determined via 
Proposition~\ref{PropPrymJacobianElliptic}, and the associated $\lambda$ and $d$
are computed by applying a linear fraction transformation to put the curve $E'$
from Proposition~\ref{PropPrymJacobianElliptic} into Legendre form.
\end{proof}

\begin{remark}
Suppose an element $\alpha\in\Aut(Z)$ takes $U\in\Jac(Z)[2]$ to $V$. If $\alpha$
does not fix the base point $(\zeta,0)$ by which we embedded $Z$ into $\Jac(Z)$,
then $\alpha$ does not necessarily provide a $K$-rational isomorphism between
$E_U$ and $E_V$, because the base point determines the appropriate twist of the
elliptic curve associated to a $2$-torsion point. We see this, for example, in
Orbits 4A and 4B: each of these orbits has two different values of $d$.
\end{remark}

\begin{remark}
\label{rmk:1mzIsogenyCurve}
The order-$4$ automorphism of $Z$ induces an order-$4$ automorphism $\zeta$ of
$\Jac(Z)$, such that multiplication by $2$ factors as $(1-\zeta)(1+\zeta)$. A
natural object of study is the degree-$4$ cover of $Z$ whose Jacobian contains
orbits 1 and 2C; it arises as $\Ztilde^H$ when $H\colonequals\ker(1- \zeta)$. 
See Sections~\ref{sec:heuristicdoublenaive} and~\ref{sec:unexpected} for more
details.
\end{remark}

\section{Heuristics for isogenous curves} \label{sec:heuristicjustthecurves}

Let $q$ be a power of an odd prime $p$ and let $K\cong \Fq$ be a finite field of
order $q$. In this section, we consider genus-$2$ curves $Z$ over $K$ having the
property that $D_4 \subseteq \Aut(Z)$. We study unordered pairs of
non-isomorphic curves of this type whose Jacobians are isogenous to one another.
The main result of this section is Theorem~\ref{T:justthecurves}, which gives
upper and lower bounds for the number of these unordered pairs in terms of $q$.

\subsection{The moduli space of genus-\texorpdfstring{$2$}{2} curves with
            \texorpdfstring{$D_4$}{D4}-action}
Recall from Notation~\ref{N:ZZ} that $\CZ$ is the set of $K$-isomorphism classes
of objects $(Z,\epsilon)$, where $Z$ is a genus-$2$ curve over $K$ and where
$\epsilon\colon D_4 \hookrightarrow \Aut(Z)$ is an embedding. Let $\CZbar$
denote the set of $K$-isomorphism classes of genus-$2$ curves over $K$ such that
$D_4 \subseteq \Aut(Z)$, and let $\nu\colon \CZ \to \CZbar$ be the forgetful
morphism taking the object $(Z,\epsilon)$ to the curve $Z$. At the the beginning
of Section~\ref{sec:D4curves} we defined $\xi$ to be the involution of $D_4$ 
that swaps the generators $a$ and $b$. We can define an involution on $\CZ$ as
well, by sending $(Z,\epsilon)$ to $(Z,\epsilon\xi)$.

\begin{notation}
\label{N:X}
Let $\CX$ be the set of isomorphism classes of objects $(E,P)$, where $E$ is an
elliptic curve over $K$ and $P$ is a $K$-rational point of order $2$ on $E$. Two
such objects $(E_1,P_1)$ and $(E_2,P_2)$ are isomorphic if there is a
$K$-rational isomorphism $E_1\to E_2$ taking $P_1$ to $P_2$.
\end{notation}

Let $\chi$ be the involution on $\CX$ that sends a pair $(E,P)$ to the pair
$(E',P')$, where $E' = E/\langle P\rangle$ and where $P'$ is the generator of
the kernel of the dual isogeny $E'\to E$. Let $\CX_{-8} \subset \CX$ be the
subset consisting of those objects $(E,P)$ such that $E$ has a
\emph{$K$-rational} endomorphism $\beta$ with $\beta^2 = -2$ for which
$\beta(P) = 0$. Let $\CX_{-4} \subset \CX$ be the subset consisting of
those $(E,P)$ such that $E$ has a \emph{geometric} automorphism $\alpha$ 
satisfying $\alpha^2 = -1$ for which $\alpha(P) = P$. Finally, let 
$\CX' = \CX\setminus\CX_{-4}$. The involution $\chi$ on $\CX$ restricts to an 
involution on $\CX'$.

In Section~\textup{\ref{ssec:structureofJac}}, we associated to every genus-$2$
curve with $D_4$-action $(Z,\epsilon)$ an elliptic curve $E$ and a $2$-torsion
point $P$ on $E$. Thus there is a map $\mu\colon \CZ \to \CX$ that sends the
isomorphism class of $(Z,\epsilon)$ to that of $(E,P)$.

\begin{proposition}
\label{P:modulispace}
The map $\mu$ is injective and has image $\CX'$. It takes the involution 
$(Z,\epsilon) \mapsto (Z,\epsilon\xi)$ of $\CZ$ to the involution $\chi$ on 
$\CX'$. The map $\nu\colon\CZ \to \CZbar$ that sends $(Z,\epsilon)$ to $Z$ is
$2$-to-$1$, unless $(Z,\epsilon)$ is fixed by $\xi$ or, equivalently, unless 
$\mu(Z, \epsilon) \in \CX_{-8}$.
\end{proposition}

\begin{proof}
Let $(E,P) \in \CX$. By Proposition~\ref{prop:ZfromEP}, there exists
$(Z,\epsilon) \in \CZ$ such that $\mu(Z,\epsilon) = (E,P)$ if and only if there
is no automorphism $\alpha\neq\pm1$ of $E$ that fixes~$P$. Combining this with 
the observation that an automorphism $\alpha\neq\pm1$ of an elliptic curve in
characteristic not $2$ that fixes a $2$-torsion point must have order~$4$, we
find that $(E,P)$ is in the image of $\mu$ if and only if it lies in $\CX'$.

Next we show that a genus-$2$ curve with $D_4$-action $(Z,\epsilon)$ can be
recovered from its image $(E,P)$ under $\mu$. To see this, we first write down a
short Weierstrass model for $E$ such that $P$ is the point $(0,0)$. Such a model
is of the form $y^2 = x(x^2 + dx + e)$, and the model is unique up to scaling
$x$ and $y$. The coefficient $e$ is nonzero because the model is nonsingular,
and $d$ is also nonzero, because otherwise the map $(x,y)\mapsto (\zeta x,-y)$
would be an automorphism of order~$4$ that fixes~$P$, and $E$ has no such
automorphisms because $(E,P)$ lies in $\CX'$. There is a unique way to scale $x$
so that the model becomes $y^2 = cx(x^2 + fx - f)$ for $c,f\in K^\times$ with
$f\ne -4$; the value of $f$ is unique, and the value of $c$ is unique up to
squares. Replacing $x$ with $x+1$ transforms the model into 
$y^2 = c (x + 1) (x^2 + s x + 1)$, where $s = f + 2 \ne \pm 2$ and where
$P = (-1,0)$. We have shown that $(E,P)$ determines unique values of $s\in K$ 
and $c\in K^\times/K^{\times 2}$, and these values determine a unique genus-$2$
curve with $D_4$-action, via Equations~\eqref{EZstwist} and~\eqref{EZstwisteps},
so $\mu$ is injective.

Lemmas~\ref{lem:quot1} and~\ref{lem:quot2} show that if 
$\mu(Z,\epsilon) = (E,P)$ then $\mu(Z,\epsilon\xi) = (E',P')$, so $\mu$ takes
the involution $(Z,\epsilon) \mapsto (Z,\epsilon\xi)$ to~$\chi$.

The fact that $\nu$ is $2$-to-$1$ except for the objects $(Z,\epsilon)$ that are
isomorphic to $(Z,\epsilon\xi)$ follows from Lemma~\ref{L:moduli2} and its
proof. By the preceding statements, $(Z,\epsilon)$ and $(Z,\epsilon\xi)$ are
isomorphic if and only if $(E,P)$ is fixed by $\chi$, meaning that $(E,P)$ and
$(E',P')$ are isomorphic. This is true if and only if $E$ has an endomorphism
$\beta$, whose kernel is generated by $P$, such that $\beta^2 = 2u$ for a unit
$u$ in $\End(E)$. The only possibilities are that: (i) $\beta = \pm1\pm \zeta$
where $\zeta \in \Aut(E)$ with $\zeta^2 = -1$, in which case 
$(E, P) \not \in \CX'$, or (ii) $\beta$ satisfies $\beta^2 = -2$. 
\end{proof}

\begin{remark}
As the preceding proof shows, the curves $Z\in\CZbar$ that have only one
preimage in $\CZ$ correspond to the exceptional curves in Lemma~\ref{L:moduli2},
that is, the curves with $s = -6$ when $-2$ is a square. We note that when
$s=-6$ the associated elliptic curve $E$ has $j$-invariant $8000$, which is the
unique root of the Hilbert class polynomial for $\ZZ[\sqrt{-2}]$. The
endomorphisms $\beta\in\End(E)$ with $\beta^2 =-2$ kill the $2$-torsion point
$P = (-1,0)$, and these endomorphisms are $K$-rational if and only if $-2$ is a
square.
\end{remark}

\begin{remark}
We can view $\CZ$ as the coarse moduli space for objects $(Z,\epsilon)$ as in
Notation~\ref{N:ZZ}. Similarly, we can view $\CX$ as the coarse moduli space
$Y_1(2) \cong Y_0(2)$, namely the modular curve $X_1(2)$ with its cusp removed.
Then, under the embedding $\CZ \hookrightarrow Y_0(2)$ the involution associated
to $\CZ \to\CZbar$ is the Fricke involution on $Y_0(2)$. The language of moduli
spaces is not useful to us here since these are not fine moduli spaces and we 
need to keep track of the field of definition of the objects. 
\end{remark}

\subsection{Counting isogenous pairs of curves with
            \texorpdfstring{$D_4$}{D4}-action}
\label{ssec:countingcurves}            

\begin{definition}
Let $P(q)$ denote the number of unordered pairs $\{Z_1,Z_2\}$, where 
$Z_1, Z_2 \in \CZbar$ are not isomorphic to one another and $\Jac(Z_1)$ and 
$\Jac(Z_2)$ are isogenous. 
\end{definition}

The following theorem determines the rate of growth of $P(q)$ up to logarithmic
factors.

\begin{theorem}
\label{T:justthecurves}
There are constants $d_1, d_2>0$ such that for all odd prime powers $q >7$,
\[ d_1 q^{3/2} \le P(q) \le d_2 q^{3/2} (\log q)^2 (\log\log q)^4.\]
If the generalized Riemann hypothesis holds, there is a constant $d_3>0$ such
that 
\[ P(q) \le d_3 q^{3/2} (\log\log q)^6.\]
\end{theorem}

\begin{remark}
Direct calculation shows that $P(q) = 0$ for $q = 3$, $5$, and $7$, so the
hypothesis that $q > 7$ in the theorem is necessary.
\end{remark}

Before we get to the proof of Theorem~\ref{T:justthecurves}, we present some
definitions and lemmas that we will need.

Proposition~\ref{prop:JacobianDecomposition} shows that if $Z$ is a genus-$2$
curve over $K$ with $D_4$-action, then $\Jac(Z)$ is isogenous to $E^2$, where 
$E$ is an elliptic curve with a rational $2$-torsion point. Since $E$ has a 
rational point of order~$2$, $\#E(K)$ is even, and since $q$ is odd, the trace
of Frobenius for $E$ must also be even. This shows that the Weil polynomial of
$\Jac(Z)$ is of the form $(x^2-tx+q)^2$, for an even integer $t$ with 
$t^2\le 4q$.

\begin{definition} 
For each even integer $t$ with $t^2 \le 4q$, let $M(q,t)$ denote the number of
$Z \in \CZbar$ whose Weil polynomial is $(x^2-tx+q)^2$, and let $N(q,t)$ denote
the number of elliptic curves over $K$ with trace $t$.
\end{definition}

\begin{lemma}
\label{L:boundM}
For all odd prime powers $q$ and even integers $t$ with $t^2 \le 4q$ we have
$M(q,t) \le 3N(q,t).$
\end{lemma}

\begin{proof}
For each curve $Z\in\CZbar$ with Weil polynomial $(x^2 - tx + q)^2$, choose an 
embedding ${\epsilon\colon D_4\hookrightarrow\Aut(Z)}$. 
Proposition~\ref{P:modulispace} shows that $(Z,\epsilon)$ gives rise to a unique
pair $(E,P)$ with ${\trace(E) = t}$, so $M(q,t)$ is at most the number of such
pairs. Since an elliptic curve has at most three rational points of order $2$, 
we have $M(q,t)\le 3 N(q,t).$
\end{proof}

\begin{lemma}
\label{L:boundN}
There is a constant $d_4$ such that for all odd prime powers $q$ and even
integers $t$ with $t^2 \le 4q$, we have
\[N(q,t) < d_4 \sqrt{q} (\log q) (\log\log q)^2.\]
If the generalized Riemann hypothesis holds, there is a constant $d_5$ such that
for all odd prime powers $q$ and even integers $t$ with $t^2 \le 4q$, we have
\[N(q,t) < d_5 \sqrt{q} (\log\log q)^3.\]
\end{lemma}

\begin{proof}
This follows from the formulas for $N(q,t)$ found 
in~\cite[Theorem~4.6, pp.~194--196]{Schoof1987}),
combined with the bounds on Kronecker class numbers found 
in~\cite[Lemma~4.4, p.~49]{AchterHowe2017}.
\end{proof}

\begin{proof}[Proof of Theorem~\textup{\ref{T:justthecurves}}]
Clearly, 
\[
P(q)
= \sum_{\stackrel{t^2 \le 4q}{\scriptscriptstyle t \text{\ even}}}
  \binom{M(q,t)}{2} 
= \sum_{\stackrel{t^2 \le 4q}{\scriptscriptstyle t \text{\ even}}} M(q,t)^2/2 
  - \sum_{\stackrel{t^2 \le 4q}{\scriptscriptstyle t \text{\ even}}} M(q,t)/2.
\]
The number of even $t$ with $t^2 \le 4q$ is at most $2\sqrt{q} + 1$. By the
Cauchy--Schwarz inequality,
\[ 
\sum  M(q,t)^2 \ge \frac{\left( \sum M(q,t)\right)^2}{(2\sqrt{q}+1)},
\]
where each sum is over the set of even $t$ with $t^2\le 4q$. By
Lemma~\ref{L:moduli2}, the sum of the $M(q,t)$ is either $q-2$ or $q-3$, so
\[P(q) \ge \frac{(q-3)^2}{2(2\sqrt{q}+1)} - \frac{q-2}{2}.\]
From this we can show that $P(q)\ge q^{3/2}/23$ for $q\ge 17$. By direct 
computation we find that $P(9) = 6$, $P(11) = 3$, and $P(13) = 6$, so we
have $P(q)\ge q^{3/2}/23$ for all $q>7$.

To prove the upper bounds on $P(q)$, we use Lemmas~\ref{L:boundM} 
and~\ref{L:boundN} to see that
\begin{align*}
P(q) &= \sum_{\stackrel{t^2 \le 4q}{\scriptscriptstyle t \text{\ even}}}
       \binom{M(q,t)}{2} 
     < \frac{81}{2}
        \sum_{\stackrel{t^2 \le 4q}{\scriptscriptstyle t \text{\ even}}}
         N(q,t)^2\\[1ex]
     & \le \frac{81}{2} (2\sqrt{q} + 1) \begin{cases}
        d_4^2\, q (\log q)^2 (\log\log q)^4 & \text{in general;}\\[1ex]
        d_5^2\, q (\log\log q)^6 & \text{if GRH holds.}
        \end{cases}
\end{align*}
The upper bounds in the theorem follow.
\end{proof}

\subsection{Gathering data}
\label{S:DataJustTheCurves}
Proposition~\ref{P:modulispace} and the ideas in 
Section~\ref{ssec:countingcurves} allow us to quickly compute the exact value of
$P(q)$ when $q$ is not too large. The subtleties in the computation include 
computing the objects $(E,P)$ that lie in $\CX_{-8}$ and in $\CX_{-4}$, and, for
each even trace~$t$, determining the number of elliptic curves with trace $t$ 
that have exactly one point of order $2$ and the number that have exactly three
points of order~$2$. This latter question is answered by noting that an elliptic
curve with Frobenius endomorphism $\pi$ has three rational points of order $2$ 
if and only if $(\pi-1)/2$ lies in its endomorphism ring, and by noting that the
number of curves with trace $t$ and with a given endomorphism ring can be 
computed from a class number; see \cite[Theorem~4.5, p.~194]{Schoof1987}.

Further details of our method of computing $P(q)$ can be found in the comments
of the Magma code we used to do so, which can be found as supplementary material
with the arXiv version of this paper as well as on the fourth author's web page:
\url{http://ewhowe.com/papers/paper51.html}

For $15\le n \le 24$, we computed the values of $P(q)/q^{3/2}$ for the $1024$
odd prime powers $q$ closest to $2^n$. For each of these sets of $1024$ prime
powers we also computed the standard deviation and the minimum and maximum 
values of $P(q)/q^{3/2}$. These values are presented in
Table~\ref{Table:JustCurves}. From Theorem~\ref{T:justthecurves} and this data,
it seems reasonable to model $P(q)$ as growing like a constant times $q^{3/2}$.

\begin{table}[ht]
\centering
\begin{tabular}{rrcccc}
\toprule
$n$ && Mean & S.d. & Max & Min \\
\midrule
15 && 0.42025 & 0.01958 & 0.45974 & 0.38473 \\
16 && 0.42188 & 0.01949 & 0.45828 & 0.38732 \\
17 && 0.42270 & 0.01953 & 0.45792 & 0.38845 \\
18 && 0.42394 & 0.01939 & 0.45914 & 0.38996 \\
19 && 0.42406 & 0.01955 & 0.45865 & 0.39078 \\
20 && 0.42464 & 0.01917 & 0.45862 & 0.39105 \\
21 && 0.42514 & 0.01942 & 0.45851 & 0.39203 \\
22 && 0.42577 & 0.01922 & 0.45830 & 0.39248 \\
23 && 0.42527 & 0.01937 & 0.45843 & 0.39262 \\
24 && 0.42557 & 0.01938 & 0.45853 & 0.39276 \\
\bottomrule
\end{tabular}
\vskip 2ex
\caption{Data for isogenous curves. For each $n$, we give the mean, standard
deviation, and extremal values of $P(q)/q^{3/2}$, where $q$ ranges over the
$1024$ odd prime powers closest to $2^n$.}
\label{Table:JustCurves}
\end{table}

\begin{remark}
\label{rem:nonisomorphic}
The quantity $P(q)$ was defined so that it counts the number of unordered pairs
$\{Z_1,Z_2\}$ of \emph{non-isomorphic} curves with $D_4$-action and with
isogenous Jacobians, because clearly $Z_1$ and $Z_2$ will have isogenous
Jacobians if in fact they are the same curve. There's another ``easy'' way that
two curves can have isogenous Jacobians: If $Z_1$ and $Z_2$ are curves over a
proper extension $\Fq$ of $\FF_p$ that are Galois conjugates of one another,
their Jacobians will be isogenous to one another via some power of the Frobenius
isogeny. If $q = p^e$, these Galois conjugate pairs account for $\Theta(e^2 q)$ 
of all the isogenous pairs over $\Fq$, which is an increasingly small fraction 
of the value of $P(q)$ as $q\to\infty$. However, when we consider the more 
uncommon doubly isogenous pairs in later sections, we will want to specifically
exclude Galois conjugate pairs from our counts.
\end{remark}

\section{Initial heuristics and data for doubly isogenous curves} 
\label{sec:heuristicdoublenaive}

By Theorem \ref{T:justthecurves}, the number of unordered pairs
$\{Z_{1},Z_{2}\}$ of genus-$2$ curves over $K = \Fq$ with $D_4$-action and with
isogenous Jacobians is proportional to $q^{3/2}$, up to logarithmic factors. 
The frequency naturally decreases if, in addition, we require that $Z_{1}$ and
$Z_{2}$ be doubly isogenous. In this section, we present our initial heuristic
about the expected number of doubly isogenous curves over $\Fq$ and some data
that we use to test the heuristic. 

\begin{remark}
\label{rmk:basepoint}
Recall that we associate to a pointed curve $(C,P)$ the cover $\Ctilde\to C$ 
obtained as the pullback of the multiplication-by-$2$ map on $\Jac(C)$ by the
Abel--Jacobi map corresponding to the base point~$P$. We say that two pointed
curves $(C_1,P_1)$ and $(C_2,P_2)$ are doubly isogenous if (the Jacobians of)
$C_1$ and $C_2$ are isogenous and (the Jacobians of) $\Ctilde_1$ and $\Ctilde_2$
are isogenous. When we are considering a genus-$2$ curve with $D_4$-action and 
with all Weierstrass points rational, there is a natural choice for a base 
point: one of the two Weierstrass points whose stabilizer under the $D_4$ action
has size~$4$. (As we saw in Section~\ref{ss:curvescovers}, it does not matter
which of these two Weierstrass points we choose, because they give isomorphic
covers.) Throughout the rest of this paper, in accordance with 
Remark~\ref{rmk:DifferentBasepoint}, when we say without further comment that
two genus-$2$ curves with $D_4$-action are doubly isogenous, we mean 
\emph{with respect to this choice of base point.}
\end{remark}

\subsection{An initial heuristic for doubly isogenous curves} 
\label{ss:initialheuristic}
As we noted in Remark~\ref{rem:nonisomorphic}, two curves $Z_1$ and $Z_2$ over a
finite field $K$ are trivially doubly isogenous if they are Galois conjugates of
one another. This observation influences the following definition.

\begin{definition}
\label{def:delta}
Let $\delta(q)$ be the number of unordered pairs $\{Z_1, Z_2\}$ of doubly
isogenous curves over $\Fq$, where $Z_1$ and $Z_2$ are genus-$2$ curves with
$D_4$-action and all Weierstrass points rational, and where $Z_1$ and $Z_2$ are
not Galois conjugates of one another.
\end{definition}

We formulate a heuristic to estimate $\delta(q)$ that we label as ``na\"{\i}ve''
because it turns out not to match the data we gathered. Later in the paper, we
explain this discrepancy and improve the heuristic. 

\begin{fheuristic} \label{FH:DoublyIsogenous}
For a fixed odd prime power $q$, we model the double-isogeny class of $Z$ as a
six-tuple of independent random elliptic curves over $\FF_q$.
\end{fheuristic}

\begin{proof}[Justification]
By Corollary \ref{cor:jacstr}, if all of the Weierstrass points of $Z$ are
rational then $\Jac(\Ztilde)$ decomposes into a sum of $17$ elliptic curves, two
of which are $E$. The remaining $15$ elliptic curves fall into six orbits under
the action of $D_4$, as in Table~\ref{tab:orbits}. The elliptic curve in Orbit
$2C$ does not depend on~$s$.

Suppose $Z_1$ and $Z_2$ are genus-$2$ curves with $D_4$-action, lying over
elliptic curves $E_1$ and $E_2$ as in Lemma~\ref{lem:quot1}. For $Z_1$ and $Z_2$
to be doubly isogenous over $\Kbar$, there must be six geometric isogenies of
elliptic curves, one between $E_1$ and $E_2$ and an additional five for the
non-constant orbits. (These five isogenies may be between orbits with different
labels; for example, orbit 2A for one curve may be isogenous to orbit 2B for the
other. This only affects the probability that that the five isogenies exist by a
constant factor.) With positive probability, a geometric isogeny
$\Jac(\Ztilde_{1}) \sim \Jac(\Ztilde_{2})$ comes from a $K$-rational isogeny,
because all the elliptic curves have a bounded number of twists.  Thus it is
reasonable to model the double-isogeny class of $Z$ as six random elliptic
curves.
\end{proof}

Let $n$ be the number of isomorphism classes of genus-$2$ curves over $K$ with
$D_4$-action and rational Weierstrass points.  It follows from the
parametrization in terms of the variable $t$ given in
\eqref{EQ:D4withWeierstrass} that $n \asymp q$; the exact count is irrelevant
for our purposes. Choose $n$ six-tuples of random elliptic curves over $\FF_q$,
and denote them by $(E_{i,1}, \ldots E_{i,6})$ for $i \in \{1,\ldots, n\}$.
Define the set 
\[
S_q \colonequals 
\{ (i,j) : i \neq j \text{ and } 
           E_{i,a} \sim E_{j,a} \text{ for } a =1,2,\ldots, 6\}
\subseteq \{1,\ldots,n\}^2.
\]

\begin{lemma} \label{lem:ev1}
The expected value of $\# S_q$ \textup(which is the prediction of Na\"{\i}ve
Heuristic~\textup{\ref{FH:DoublyIsogenous}} for the number of pairs of doubly
isogenous curves\textup) satisfies 
\[\EE(\# S_q) \asymp  1/q.\]
\end{lemma}

\begin{proof}
There are $\Theta(q^2)$ pairs of $(i,j)$, and in 
Section~\ref{sec:heuristicjustthecurves} we showed the probability of two random
elliptic curves over $\Fq$ being isogenous is $\Theta(q^{-1/2})$.
\end{proof}

Thus Na\"{\i}ve Heuristic~\ref{FH:DoublyIsogenous} predicts that the ``expected
value'' of $\delta(q)$ is $\asymp 1/q$.  As we will see, this does not match the
data we gathered; the assumption that the $6$ elliptic curves are independent
does not turn out to be completely accurate.

\subsection{Data that does not support the na\"{\i}ve heuristic}
\label{ssec:nonsupporting}

In order to calculate $\delta(q)$ for specific values of $q$, we first want to
enumerate all isomorphism classes of genus-$2$ curves $Z$ over $K\cong\Fq$ with
$D_4$-action and with all Weierstrass points rational. To do so, we vary $t$ 
in~\eqref{EQ:D4withWeierstrass}. Letting $s = -(t^4+1)/t^2$ and taking into
account the involution $s\leftrightarrow s'$, we see that the following values
of $t$ give isomorphic curves:
\begin{equation} \label{eq:choicest}
    \pm t,                             \quad  
    \pm \biggl(\frac{1}{t}\biggr),     \quad 
    \pm \biggl(\frac{t-1}{t+1}\biggr), \quad 
    \pm \biggl(\frac{t+1}{t-1}\biggr).
\end{equation}
To enumerate isomorphism classes, we fix an ordering of the elements of~$K$ and 
only consider values of $t$ for which $s \neq \pm 2$ (so $Z$ is non-singular) 
and for which $t$ is the smallest of the values in~\eqref{eq:choicest}. We then
include the curve $Z$ from \eqref{EQ:D4withWeierstrass} and its standard 
quadratic twist in our enumeration. If $w$ is a fixed quadratic non-residue 
of~$K$, that means we look at
\[
    y^2 =  (x^2 + 1)(x^4 + s x^2 + 1)\textup{\quad and\quad}
    y^2 = w(x^2 + 1)(x^4 + s x^2 + 1).
\]
(Recall from Lemma~\ref{L:moduli2} that if $s=-6$ and $-2$ is not a square 
in~$K$, the curve $Z$ is isomorphic to its standard quadratic twist. However,
when $s = -6$ we have $t = \pm 1\pm \sqrt{2}$, so when the Weierstrass points of
$Z$ are rational, the exceptional case in Lemma~\ref{L:moduli2} does not occur.)

\begin{table}[ht]
\centering
\begin{minipage}{.5\textwidth}
  \centering
  \begin{tabular}{c r@{}r}
    \toprule
  $n$ & \multicolumn{2}{c}{Examples} \\
\midrule
15 & \hbox to 1.5em{} 820 & \\
16 &                  580 & \\
17 &                  407 & \\
18 &                  282 & \\
19 &                  218 & \\
20 &                  138 & \\
21 &                  100 & \\
22 &                   58 & \\
23 &                   42 & \\
 \bottomrule
\end{tabular}
\vskip 2ex
  \caption*{A. Data for doubly isogenous curves
            \phantom{filler}}
\end{minipage}
\begin{minipage}{.49\textwidth}
  \centering
    \begin{tabular}{c r}
    \toprule
  $n$ & \multicolumn{1}{c}{Examples} \\
\midrule
15 &   11690160 \\ 
16 &   23837994 \\ 
17 &   48443688 \\ 
18 &   97608276 \\ 
19 &  196343212 \\ 
20 &  394584130 \\ 
21 &  793839836 \\ 
22 & 1588282776 \\ 
23 & 3172154548 \\ 
 \bottomrule
\end{tabular}
\vskip 2ex
  \caption*{B. Data for $[1-\rho^*]$-isogenous curves 
            (see Example~\ref{ex:heurzeta})}
\end{minipage}
\caption{
The total number of unordered pairs of doubly isogenous curves and
$[1-\rho^*]$-isogenous curves over $\Fq$ for the $1024$ primes
$q\equiv 1\bmod 4$ closest to $2^n$, restricting to curves with all Weierstrass 
points rational.}
\label{Table:TwoTables}
\end{table}

In Table~\ref{Table:TwoTables}(A), we present some data that we collected by
enumerating doubly isogenous pairs. For $n$ ranging from $15$ to $23$, we
considered the $1024$ primes $q\equiv 1\bmod 4$ closest to $2^n$ and computed
the sum of $\delta(q)$ over these values of $q$. According to Na\"{\i}ve
Heuristic~\ref{FH:DoublyIsogenous} and Lemma~\ref{lem:ev1}, we would expect this
sum to have rate of growth of the form $c/2^n$ for some constant $c$. In
particular, we would expect the sum to approximately halve as we increase $n$
by~$1$. This is not what we observe. We explain this discrepancy in the next
section by finding several families of coincidences that cause doubly isogenous
pairs to occur more often than predicted.  

We can similarly develop heuristics for covers corresponding to subgroups of
$\Jac(Z)[2]$. 

\begin{definition}
Let $H$ be a subgroup of $\Jac(Z)[2]$.  We say that $Z_1$ and $Z_2$ are
\emph{$H$-isogenous} if $\Jac(Z_1)$ and $\Jac(Z_2)$ are isogenous and
$\Jac(\widetilde{Z}_1^{H})$ and $\Jac(\widetilde{Z}_2^H)$ are isogenous, where
$\widetilde{Z}_1^{H}$ and $\widetilde{Z}_2^H$ are as defined in 
Definition~\ref{DpiH}.
\end{definition}

(When $H=0$ we recover the definition of doubly isogenous curves.)

Proposition~\ref{prop:decomposecover} gives a decomposition of 
$\Jac(\Ztilde^H)$, and Table \ref{tab:orbits} lets us identify the elliptic
curves appearing in this decomposition. In particular, if $m$ is the number of
different non-constant orbits of elliptic curves corresponding to the
$2$-torsion points in $H^\perp$, then we expect 
\begin{equation} \label{EheuristicdoubleH}
\#\{ H \text{-isogenous pairs } /\Fq  \} \asymp q^{(3-m)/2} .
\end{equation}

\begin{example} \label{ex:heurzeta}
For example, take $H \colonequals \ker(1-\rho^*)$, where $\rho$ is the 
automorphism of order~$4$ defined in Section~\ref{ssec:automorphisms}. The 
cover $\Ztilde^H\to Z$ has degree four, and since $(1-\rho^*)^2 = -2\rho^*$, it is 
isomorphic to the pullback of the endomorphism $1 - \rho^*$ on $\Jac(Z)$ via the
embedding $Z \to \Jac(Z)$.  When $Z_1$ and $Z_2$ are $H$-isogenous for this $H$,
we say that $Z_1$ and $Z_2$ are \emph{$[1-\rho^*]$-isogenous}. The Jacobian
$\Jac(\Ztilde^H)$ contains orbits 1 and 2C (in addition to~$E^2$). Since 
Orbit~2C is constant, $m=1$ and we only need a single coincidence for
$\Jac(\Ztilde_{1}^H)$ and $\Jac(\Ztilde_{2}^H)$ to be isogenous.  Thus we 
expect 
\begin{equation} \label{Eheurzeta}
\#\{[1-\rho^*]\text{-isogenous pairs } /\Fq \} \asymp q .
\end{equation}
We expect the total number of pairs of $[1-\rho^*]$-isogenous curves for the
$1024$ primes $q\equiv 1\bmod 4$ closest to $2^n$ to have rate of growth 
$c \cdot 2^n$ for some constant $c$, and this is supported by the data in 
Table~\ref{Table:TwoTables}(B), as the number of pairs roughly doubles as we
increase $n$ by~$1$.  
\end{example}

\section{Families with unexpected coincidences} \label{sec:unexpected}
Na\"{\i}ve Heuristic~\ref{FH:DoublyIsogenous} predicts that the ``expected
value'' of $\delta(q)$ is on the order of $1/q$, where $\delta(q)$ is the number
of non-conjugate pairs $\{Z_1, Z_2\}$ of doubly isogenous curves over a finite
field $\Fq$ with $Z_1$ and $Z_2$ genus-$2$ curves with $D_4$-action and all
Weierstrass points rational. As seen in Table~\ref{Table:TwoTables}(A), the data
we collected does not seem to reflect this rate of growth. In this section, we
find a number of families of coincidences that explain this discrepancy and we
formulate a more sophisticated heuristic for the number of such pairs, which
will be supported by the data in Section~\ref{ss:heuristicsfamily}.

\subsection{\texorpdfstring{$j$}{j}-invariants for orbits}
\label{ssec:t}

We begin by computing the $j$-invariants of the elliptic curves appearing in
Table~\ref{tab:orbits}; the middle column of Table~\ref{tab:orbitjs} gives these
$j$-invariants in terms of the parameter $t$. For our computations, it will be
convenient to note that we can also express these $j$-invariants in terms of the
quantity 
\[
    u\colonequals (1/2)(t - 1/t),
\]
as is shown in the third column of Table~\ref{tab:orbitjs}. This new
parametrization simplifies our computations in 
Section~\ref{ss:findingisogenies}, because $I = (u + 1/u)^2$ is a quartic
function of $u$ instead of a degree-$8$ function of $t$. We omit the proofs of
the following two facts. 

\begin{table}[ht]
\centering
\begin{tabular}{l@{\qquad}c@{\qquad}c}
\toprule
      & $j$-invariant, in terms & $j$-invariant, in terms\\
Orbit & of the variable $t$     & of the variable $u$    \\
\midrule
 & \\[-2ex] 
\quad 1  & $\displaystyle\frac{ 2^4 (t^8 + 14t^4 + 1)^3}{(t^5 - t)^4}$                                 & $\displaystyle\frac{ 2^8 (u^4 + u^2 + 1)^3}{ u^4 (u^2 + 1)^2}$    \\[3ex]
\quad 2A & $\displaystyle\frac{-2^4 (t^4 - 14t^2 + 1)^3}{t^2(t^2 + 1)^4}$                              & $\displaystyle\frac{-2^6 (u^2 - 3)^3}{(u^2 + 1)^2}$               \\[3ex]
\quad 2B & $\displaystyle\frac{ 2^6 (3t^4 - 10t^2 + 3)^3}{(t^2 - 1)^2 (t^2 + 1)^4}$                    & $\displaystyle\frac{ 2^6 (3u^2 - 1)^3}{(u^3 + u)^2}$              \\[3.5ex]
\quad 2C & $1728$                                                                                      & $1728$                                                            \\[2ex]
\quad 4A & $\displaystyle\frac{-2^6(t^4 - 2\zeta t^3 - 6 t^2 + 2\zeta t + 1)^3}{(t^3-t)^2(t-\zeta)^4}$ & $\displaystyle\frac{-2^8 (u^2 - \zeta u - 1)^3}{(u^2-\zeta u)^2}$ \\[3ex]
\quad 4B & $\displaystyle\frac{-2^6(t^4 + 2\zeta t^3 - 6 t^2 - 2\zeta t + 1)^3}{(t^3-t)^2(t+\zeta)^4}$ & $\displaystyle\frac{-2^8 (u^2 + \zeta u - 1)^3}{(u^2+\zeta u)^2}$ \\[2ex]
\bottomrule
\end{tabular}
\vskip 2ex
\caption{The $j$-invariants for the elliptic curves in Table~\ref{tab:orbits} in 
terms of $t$ and $u = \frac{1}{2}(t - 1/t)$.}
\label{tab:orbitjs}
\end{table}

\begin{lemma}
\label{rem:changeu}
Replacing $u$ with $-u$, $1/u$, or $-1/u$ does not change the value of $I$.
\qed
\end{lemma}

\begin{lemma}
\label{rem:BaseCurve}
The elliptic curve $E$ defined by Equation~\eqref{EdefE} has $j$-invariant 
\[
\frac{2^8(t^4 - t^2 + 1)^3}{t^4 (t^2 - 1)^2} 
 = \frac{2^6 (4 u^2 + 1)^3 }{u^2},
 \]
and is isomorphic to $y^2 = dx(x-1)(x-\lambda)$, where $\lambda = t^2$ and 
$d = c(t^2 + 1)$. \qed
\end{lemma}

\subsection{Finding generic geometric isogenies} 
\label{ss:findingisogenies} 

In this subsection, we work over an algebraically closed field $K$ of
characteristic not~$2$. Our goal is to find families of ordered pairs
$(Z_1,Z_2)$ of genus-$2$ curves over $K$ with $D_4$-action that have a
higher-than-expected chance of being doubly isogenous. For example, we might
search for families where the elliptic curves in some of the orbits listed in
Table~\ref{tab:orbits} for $Z_1$ are automatically isogenous to those in some of
the orbits for $Z_2$.  We carry out this search by using classical modular
polynomials $\Phi_n\in\ZZ[x,y]$.  (Recall that the polynomial $\Phi_n$ has the
property that there is a geometric cyclic $n$-isogeny between two elliptic
curves over an arbitrary field $K$ if and only if the $j$-invariants $j_1$ and
$j_2$ of the two curves satisfy $\Phi_n(j_1,j_2) = 0$; see \cite{Igusa1959},
\cite{Igusa1968}.)

Let $Z_1$ and $Z_2$ be two genus-$2$ curves over $K$ with $D_4$-action, and let
$I_1$ and $I_2$ be their respective invariants; see Section~\ref{sec:invariant}. 
We can write each $Z_i$ in the form~\eqref{EQ:D4withWeierstrass}; that is, there
are $c_i, t_i\in K$ such that $Z_i$ is given by
\begin{align} \label{eq:zi}
Z_i \colon\quad y^2 &= c_i (x-\zeta)(x+\zeta)(x-t_i)(x+t_i)(x-1/t_i)(x+1/t_i)\\
                    &= c_i (x^2+1)(x^4+s_i x^2+1), \notag
\end{align}
where $s_i = -(t_i^4+1)/t_i^2$. Since $K$ is algebraically closed, we may take
$c_1 = c_2 = 1$.

As we observed in Section~\ref{ssec:t}, the orbit labels for $Z_1$ and $Z_2$ 
are determined by the values of $u_i = (1/2)(t_i - 1/t_i)$, which satisfy 
$I_i = (u_i + 1/u_i)^2$. It follows that $u_1$ and $u_2$ are reasonable
parameters to use for the families we construct. 

To find families of ordered pairs $(Z_1,Z_2)$ with a cyclic isogeny of degree
$n$ between specified orbits, we work over the algebraic closure $K$ of the
$2$-variable function field $\QQ(\zeta)(u_1,u_2)$, and consider the curves $Z_1$
and $Z_2$ with parameters $t_1, t_2 \in K$ such that $u_i = (1/2)(t_i - 1/t_i)$.
Given an orbit for $Z_1$ and an orbit for $Z_2$, we can plug the appropriate
formulas for the $j$-invariants of the orbits into $\Phi_n$ in order to obtain
an expression in $u_1$ and $u_2$ which is zero if and only if there is a cyclic
$n$-isogeny between the curves in the given orbits.

\begin{example} \label{ex:11}
Let us calculate conditions under which the Orbit 1 elliptic curve for $Z_1$ is
geometrically isomorphic to the Orbit 1 elliptic curve for $Z_2$. This 
calculation is simpler than most, because the $j$-invariants of the Orbit 1
elliptic curves can in fact be expressed directly in terms of the invariants of
$Z_1$ and $Z_2$; namely, the Orbit 1 $j$-invariant for each curve is 
$256 \cdot (I_i- 1)^3/I_i$. We see that the two $j$-invariants are equal if and
only if 
\[
(I_1-1)^3 I_2 - (I_2-1)^3 I_1 = 0.
\]
The expression on the left-hand side factors as the product of 
$I_1^2 I_2 + I_1 I_2^2 - 3 I_1 I_2 + 1$ and $I_1 - I_2$.

We compute that the condition $I_1^2 I_2 + I_1 I_2^2 - 3 I_1 I_2 + 1 = 0$ is
equivalent to
\[
(u_1^2 + u_2^2 + 1 )
(u_1^2 u_2^2 + u_1^2 + 1 )
(u_1^2 u_2^2 + u_2^2 + 1 )
(u_1^2 u_2^2 + u_1^2 + u_2^2) = 0\,.
\]
\end{example}

\begin{example} \label{ex:2a2b}
Let us consider the relation between $u_1$ and $u_2$ that is satisfied exactly 
when the elliptic curves in Orbit 2A of $Z_1$ are $2$-isogenous to the elliptic
curves in Orbit 2B of~$Z_2$. We will not write down the full polynomial relation
in $u_1$ and $u_2$ because it involves $94$ terms. We do observe that it factors
over $\QQ(\zeta)$ into nine irreducible polynomials; one of these irreducible
factors is $u_1^2 u_2^2 + u_1^2 + 1$, which also appears in Example~\ref{ex:11}!
Thus, if the single relation $u_1^2 u_2^2 + u_1^2 + 1 = 0$ holds, the Orbit 1
curve for $Z_1$ is isomorphic to the Orbit 1 curve for $Z_2$, \emph{and} the
Orbit 2A  curve for $Z_1$ is $2$-isogenous to the Orbit 2B curve for $Z_2$. 
This is an unexpected coincidence!
\end{example}

In the following subsection we report on what we found by systematically
searching for such coincidences. The ideal but computationally intensive
calculation would be to work over $K$ and consider every pair $(F_1,F_2)$ of
elliptic curves, where each $F_i$ is either the quotient of $Z_i$ given 
by~\eqref{EdefE} or one of the curves in an orbit for $Z_i$. For each positive
integer $n$ in some a predetermined set of values (see
Remark~\ref{rmk:isogenydegrees} for our choice), we would compute an expression
in $u_1$ and $u_2$ that equals zero if and only if there is a cyclic $n$-isogeny
between the curves in the two orbits. Then, for every pair of such expressions,
we would compute their greatest common divisor. Whenever this greatest common
divisor was not $1$, we would find a family of pairs $(Z_1,Z_2)$ of curves
associated to a pair of parameters $(u_1,u_2)$ where there are multiple
isogenies between the elliptic factors of $\Jac(\Ztilde_1)$ and those of
$\Jac(\Ztilde_2)$.

It is computationally difficult to implement the above strategy because when we
substitute the rational functions for the $j$-invariants into all but the
smallest modular polynomials, the expressions become quite large. To reduce the
size of the coefficients in the expressions, and to reduce the number of
monomials involved, we instead work modulo a prime $p\equiv 3\bmod 4$ and 
specialize $u_1$ to a value in $\FF_p(\zeta)$. This yields rational functions in
$\FF_p(\zeta)(u_2)$ which fit much more easily in memory. We are therefore 
looking for cyclic $n$-isogenies between fibers (over $u_2$) of the family $Z_2$
and a fixed fiber of~$Z_1$, after reducing modulo $p$.  

There are two risks associated to making these reductions.  The first is that we
may find spurious relations, nonzero greatest common divisors that occur only
modulo $p$. As it happens, none of the relations we found involved modular
polynomials of high degree, so we were subsequently able to verify the relations
we found over the full ring $\QQ(\zeta)(u_1,u_2)$.

The second risk is that we might miss a family. This could happen, for example,
if there is a relation that involves a polynomial in $\QQ(\zeta)(u_1,u_2)$ that
reduces modulo $p$ to a constant, or to a polynomial like $u_1u_2$ whose
solutions require one of the $u_i$ to be equal to one of the forbidden values
$0$ or $\zeta$. It could also happen if we specialize to a value of $u_1$ that
makes the polynomial constant. Without knowing more about the geometry of the
possible families in characteristic zero, we are not sure how to rule out these
possibilities. We did, however, run our computation several times, with
different choices for the prime $p$ and different choices for the values of
$u_1$, and the results did not vary. Thus, we believe we found all of the
families of coincidences involving isogenies of the degrees we considered. As we
will see in Section~\ref{ss:heuristicsfamily}, we have found enough families to
formulate an improved heuristic that is supported by our data.

\begin{remark}
\label{rmk:isogenydegrees}
We are left to specify the degrees $n$ of the cyclic isogenies we will consider.
We choose to look for cyclic $n$-isogenies for all values of $n$ for which the
modular curve $X_0(n)$ has genus $0$ 
(namely, $n=1, 2, 3, 4, 5, 6, 7, 8, 9, 10, 12, 13, 16, 18,$ and $25$)
or genus $1$ 
(namely, $n=11, 14, 15, 17, 19, 20, 21, 24, 27, 32, 36,$ and~$49$). 
We chose these values so that we would find all families of coincidences given
by a relation between $u_1$ and $u_2$ that defines a curve of geometric genus at
most~$1$. To see that these values of $n$ will lead to all such families, note
that every family we find gives us a varying pair of elliptic curves connected
by a cyclic $n$-isogeny, and so comes provided with a nonconstant map to
$X_0(n)$. Since no family of genus $0$ or $1$ can map to a modular curve with
genus larger than $1$, for our goal of finding all families defined by genus-$0$
and genus-$1$ relations between $u_1$ and $u_2$, it will suffice for us to 
consider the values of $n$ specified above.
\end{remark}

In the end, we found sixteen families of coincidences in terms of $u_1$ and
$u_2$. However, if we count two families as being equivalent if they produce the
same pairs $(Z_1,Z_2)$ --- that is, if one family can be obtained from the other
by applying transformations from Remark~\ref{rem:changeu} to $u_1$ and $u_2$ ---
then we have only four equivalence classes of families. We describe these four 
classes of families in the following section, where we keep close track of 
fields of definition of isogenies.

\begin{remark}
A family given by a relation between $u_1$ and $u_2$ can be made more explicit
by replacing each $u_i$ with $(1/2)(t_i - 1/t_i)$, and then looking at an
irreducible factor of the resulting expression. For example, the relation
between $t_1$ and $t_2$ obtained from the relation $u_1^2 + u_2^2 + 1 = 0$ from
Example~\ref{ex:11} has degree $6$, but it factors into two factors of degree
$1$ and two factors of degree $2$. One of the factors is $t_2 - \zeta t_1$. See
Section~\ref{Subsubsection:FirstFamily}.
\end{remark}

\subsection{Description of the families} 
\label{ssec:families}
\subsubsection{The first family}
\label{Subsubsection:FirstFamily}
Let $K$ be a field of characteristic not $2$ that contains a primitive $4$th
root of unity~$\zeta$. For $i=1,2$, let $c_i$ and $t_i$ be elements of
$K^\times$ with $t_i^4 \ne 1$, let $Z_i$ be given by~\eqref{eq:zi}, and let
$E_i$ be the quotient curve
\begin{equation}
\label{eq:ei}
E_i\colon \quad y^2 = c_i (x + 1)(x - t_i^2)(x - 1/t_i^2).
\end{equation}

\begin{lemma} \label{lem:family1}
Suppose $t_2 = \zeta t_1$, and suppose $(t_1^2+1)(t_2^2+1)$ and $c_1 c_2$ are
squares in~$K$. Then over $K$ the following statements hold\textup{:}
\begin{enumerate}
\item the elliptic curve in Orbit 1 for $Z_1$ is isomorphic to the elliptic
      curve in Orbit 1 for~$Z_2$\textup{;}
\item the elliptic curves in Orbit 2B for $Z_1$ and the elliptic curves in 
      Orbit~2B for $Z_2$ are related by a degree-$2$ isogeny\textup{;}
\item the elliptic curves in Orbit 2C for $Z_1$ are isomorphic to those in 
      Orbit~2C for $Z_2$\textup{;}     
\item the elliptic curve $E_1$ \textup(resp.\ $E_2$\textup) and the elliptic
      curves in Orbit 2A for $Z_2$ \textup(resp.\ $Z_1$\textup) are related by
      a degree-$2$ isogeny.
\end{enumerate}
Furthermore, if $K$ is the algebraic closure of the function field $\QQ(t)$ and
$t_1 = t$, then there are no other isogenies among the orbits associated to
$Z_1$ and~$Z_2$.
\end{lemma}

The Lemma could also be rephrased in terms of Prym varieties using 
Proposition~\ref{PropPrymJacobianElliptic}.

\begin{proof}[Proof of Lemma~\ref{lem:family1}]
To avoid a proliferation of subscripts and to aid in visual comprehension of 
various formulas, in this proof we will write $t$ and $c$ for $t_1$ and $c_1$,
and we will write $T$ and $C$ for $t_2$ and $c_2$. More generally, we will use
lower case letters for variables associated with $Z_1$, and upper case letters
for variables associated with $Z_2$.

Since the isomorphism classes of the curves $Z_i$ and the elliptic curves $E_i$
only depend on the values of $c$ and $C$ up to squares, and since $cC$ is a
square by hypothesis, we may assume that  $C = c$.

By Table~\ref{tab:orbits}, the Orbit 1 curves for $Z_1$ and $Z_2$ can be written
as
\[ 
   y^2 = x(x-1)(x-\lambda) \text{\quad and\quad}
   Y^2 = X(X-1)(X-\Lambda),
\]
respectively, where $\lambda  = 4t^2/(t^2 + 1)^2$ and 
$\Lambda = 4T^2/(T^2 + 1)^2 = -4 t^2 / (t^2 - 1)^2$. Then an isomorphism between
the curves is given by
\[
    X = (x - \lambda)/(1 - \lambda) \text{\quad and\quad}
    Y = y \cdot (t^2 + 1)^3/(t^2 - 1)^3.
\]
This proves the first statement.

The Orbit 2B curves for $Z_1$ and $Z_2$ can be written as
\[ y^2 = c (t^2+1) \cdot x(x-1)(x-\lambda) \text{\quad and\quad}
   Y^2 = c (T^2+1) \cdot X(X-1)(X-\Lambda),
\]
respectively, where 
\[
\lambda  = \frac{2(t^2-1)}{(t + \zeta)^2} \text{\quad and\quad}
\Lambda  = \frac{2(T^2-1)}{(T + \zeta)^2} = \frac{2(t^2+1) }{ (t + 1)^2}.
\]
Then a degree-$2$ isogeny from the first curve to the second is given by
\begin{align*}
    X &= \frac{\zeta}{2(t+1)^2} \cdot
         \frac{((t + \zeta) x - (1 + \zeta)(t + 1))^2}{x - 1}\\
    Y &= y \cdot \frac{(t + \zeta) (\zeta - 1) }{ 4 (t + 1)^3}
           \cdot \frac{(t + \zeta)^2 x^2 - 2(t + \zeta)^2 x + (2t^2 - 2)}
                      {(x-1)^2}.
\end{align*}
This proves the second statement.

The Orbit 2C curves for $Z_1$ and $Z_2$ are both twists of $y^2 = x^3 - x$, by
$c(t^2+1)$ and by $c(T^2+1)$, respectively.  By hypothesis, $(t^2+1)(T^2+1)$ is
a square in $K$, so these two twists are isomorphic to one another. This proves
the third statement.

By Remark~\ref{rem:BaseCurve} and Table~\ref{tab:orbits}, we can write $E_1$ and
the Orbit 2A curve for $Z_2$ as 
\[ y^2 = c(t^2+1) \cdot x(x-1)(x-\lambda) \text{\quad and\quad}
   Y^2 = c(T^2+1) \cdot X(X-1)(X-\Lambda),
\]
respectively, where 
$\lambda = t^2$ and $\Lambda = 4\zeta T/(T + \zeta)^2 = 4t/(t+1)^2$.
Then a degree-$2$ isogeny from the first curve to the second is given by
\[
X = \frac{1}{(t+1)^2} \cdot \frac{(x + t)^2}{x} \text{\quad and\quad}
Y = y \cdot \frac{1}{(t + 1)^3} \cdot \frac{x^2 - t^2}{x^2} .
\]
This proves one case of the fourth statement. The proof of the other case is
similar.

Finally, we check that no other pairs of elliptic curves associated to $Z_1$ and
$Z_2$ are isogenous to one another when $K$ is the algebraic closure of $\QQ(t)$
and $t_1 = t$. As we noted earlier, two elliptic curves over a field are
connected by a cyclic $n$-isogeny over the algebraic closure if and only if
their $j$-invariants satisfy the $n$-th classical modular polynomial
$\Phi_n\in\ZZ[x,y]$. If this relation holds in $K$, then it will also hold when
we reduce modulo $p$ and specialize $t$ and $c$ to specific values for which the
resulting curves are nonsingular. We take $p = 421$, $\zeta = 29\in\FF_p$, 
$t=19\in\FF_p$, and $c = 1\in\FF_p$. Computing traces shows the elliptic curves
in question are all ordinary. It follows that any geometric isogeny between them
is defined already over $\FF_{p^{12}}$; see 
\cite[\S{5}, p. 251]{HoweNartRitzenthaler2009}. Thus we simply compute the 
traces of these elliptic curves over $\FF_{p^{12}}$ and observe that there are
no further matches.
\end{proof}

We noted in Section~\ref{ssec:nonsupporting} that replacing $t_1$ with any of
eight linear fractional expressions in $t_1$ will result in a curve isomorphic
to $Z_1$, but possibly with the labels on Orbits 2A and 2B swapped, and
similarly for Orbits 4A and 4B. If we take the relation $t_2 = \zeta t_1$ and
apply one of these transformations to $t_1$ and a possibly different one to
$t_2$, we will get another family of curves satisfying the conclusions of
Lemma~\ref{lem:family1}, possibly with the roles of various orbits swapped.
There are $64$ ways of applying these eight linear fractional transformations
separately to $t_1$ and $t_2$, but some of these will produce equivalent
relations; for example, replacing $t_1$ with $-t_1$ and $t_2$ with $-t_2$ fixes
the relation $t_2 = \zeta t_1$. In fact, we obtain only $16$ different families
in this way. When we multiply the $16$ corresponding relations together, we get
a relation that can be expressed in terms of $I_1$ and $I_2$, namely:
\begin{equation} \label{eq:family1}
    I_1^2 I_2 + I_1 I_2^2 - 3I_1 I_2 + 1 = 0.
\end{equation}

\begin{definition}
We say that two curves $Z_1$ and $Z_2$ with $D_4$-action are
\emph{in the first family} if their invariants satisfy~\eqref{eq:family1}.
\end{definition}

\begin{remark}
Equation \eqref{eq:family1} defines a genus-$0$ curve, which can be parametrized
as
\[ I_1 = \frac{-(1+z)^2}{2(1-z)}, \quad I_2 = \frac{-(1-z)^2}{2(1+z)}.\]
Under this parametrization, the involution swapping $I_1$ and $I_2$ corresponds
to ${z\leftrightarrow -z}$.
\end{remark}

Proposition~\ref{PropPrymJacobianElliptic} shows that each of the elliptic
curves that appears in Table~\ref{tab:orbits} as an isogeny factor of
$\Ztilde_i$ can also be viewed as a Prym variety $\Prym^H$ for a double cover of
$Z_i$ specified by an index-$2$ subgroup $H$ of $\Jac(Z_i)[2]$, and each such
subgroup $H$ is determined as the set of elements of $\Jac(Z_i)[2]$ that pair
trivially with a nonzero element $U\in\Jac(Z_i)[2]$. Lemma~\ref{lem:family1} 
could therefore be restated in terms of these Pryms. The labeling of these Pryms
via the elements $U$ has the same problem as the labeling of the orbits 
associated to the $Z_i$: The labels depend on which of the eight possible values
of $t_i$ we used to write down an equation for~$Z_i$. However, using Prym 
varieties we can state a variant of Lemma~\ref{lem:family1} whose hypotheses and
conclusions depends only on the isomorphism classes of the curves $Z_i$ and not
on the choices we made to write them down.

\begin{notation}
Let $Z$ be a genus-$2$ curve with $D_4$-action over a field $K$ of 
characteristic not~$2$. As we see from Table~\ref{tab:orbits}, there is a 
unique nonzero point $U$ of $\Jac(Z)[2](\Kbar)$ that is fixed by the action 
of~$D_4$. Let $H$ be the order-$2$ subgroup of $\Jac(Z)[2]$ generated by~$U$. 
In the notation of Definition~\ref{DpiH}, let $\Zhat = \Ztilde^H$, so that there
is a degree-$8$ cover $\Zhat\to Z$ and $\Zhat$ has genus~$9$.
\end{notation}

\begin{proposition} \label{prop:intrinsiccover}
Let $Z_1$ and $Z_2$ be curves with $D_4$-action over a field $K$ of 
characteristic not~$2$, and suppose $Z_1$ and $Z_2$ lie in the first family. If 
$\Jac(Z_1)$ and $\Jac(Z_2)$ are geometrically isogenous to one another, then 
$\Jac(\Zhat_1)$ and $\Jac(\Zhat_2)$ are geometrically isogenous to one another.
\end{proposition}

\begin{proof}
We may assume that $K$ is algebraically closed. Let $\zeta$ be a primitive
fourth root of unity in $K$. Since $Z_1$ and $Z_2$ are in the first family and
since $K$ is algebraically closed, we can choose values of $t_1$ and $t_2$ with 
$t_2 = \zeta t_1$ such that each $Z_i$ has a model as
in~\eqref{EQ:D4withWeierstrass} with $t = t_i$ and with $c = 1$. The hypotheses
of Lemma~\ref{lem:family1} are then satisfied. 

Since $\Jac(Z_1)$ and $\Jac(Z_2)$ are isogenous to one another and since
$\Jac(Z_i)\sim E_i^2$ for each~$i$, the elliptic curves $E_1$ and $E_2$ are
isogenous to one another. Lemma~\ref{lem:family1} then shows that the Orbit 1
curves for $Z_1$ and $Z_2$ are isogenous to one another, as are the Orbit 2A
curves, the Orbit 2B curves, and the Orbit 2C curves. 

Proposition~\ref{prop:decomposecover} shows that $\Jac(\Zhat_i)$ decomposes up 
to isogeny as the product of $E_i^2$ with the product of the $\Prym^{H'}$, where
$H'$ ranges over the index-$2$ subgroups of $\Jac(Z_i)[2]$ that contain the
subgroup $H = \langle U\rangle$, where $U = [W_\zeta-W_{-\zeta}]$. Taking duals
with respect to the Weil pairing, we see that the $H'$ are the index-$2$ 
subgroups such that $(H')^\perp\subseteq H^\perp$. This means that the 
$(H')^\perp$ are precisely the subgroups $\langle U'\rangle$, where $U'$ is a
nontrivial $2$-torsion point that pairs trivially with $U$. We see from 
Lemma~\ref{lem:explicitweil} and Table~\ref{tab:orbits} that these $U'$ are the
labels of Orbits 1, 2A, 2B, and 2C, so each $\Jac(\Zhat_i)$ is isogenous to
$E_i^2$ times the product of the elliptic curves in Orbits 1, 2A, 2B, and 2C.
Therefore, the $\Jac(\Zhat_i)$ are isogenous to one another.
\end{proof}

\begin{proposition} \label{prop:family1t}
Let $Z_1$ and $Z_2$ be curves with $D_4$-action given as in~\eqref{eq:zi} by
values of $t_i$ and $c_i$ such that $t_2 = \zeta t_1$ and such that $c_1 c_2$
and $(t_1^2+1)(t_2^2+1)$ are both squares. For $i=1,2$, let $E_i$ be the
elliptic curve given by~\eqref{eq:ei}. Suppose that $E_1$ and $E_2$ are
isogenous to one another, and that either of the following two conditions
holds\textup{:}
\begin{enumerate}
\item The curve 
      $y^2 =                     \zeta x(x-1)(x+2\zeta t_1/(t_1-\zeta)^2)$ is isogenous to either 
      \begin{align*}
      y^2 &=                     \zeta x(x-1)(x+2\zeta t_2/(t_2-\zeta)^2) \text{\quad or}\\
      y^2 &= \zeta c_2 (t_2^2+1) \cdot x(x-1)(x+2\zeta t_2/(t_2-\zeta)^2),
      \end{align*}
      and the curve 
      $y^2 =                    \zeta  x(x-1)(x-2\zeta t_1/(t_1+\zeta)^2)$ is isogenous to either 
      \begin{align*}
      y^2 &=                     \zeta x(x-1)(x-2\zeta t_2/(t_2+\zeta)^2) \text{\quad or}\\
      y^2 &= \zeta c_2 (t_2^2+1) \cdot x(x-1)(x-2\zeta t_2/(t_2+\zeta)^2).
      \end{align*}
\item The curve 
      $y^2 =                     \zeta x(x-1)(x+2\zeta t_1/(t_1-\zeta)^2)$ is isogenous to either 
      \begin{align*}
      y^2 &=                     \zeta x(x-1)(x-2\zeta t_2/(t_2+\zeta)^2) \text{\quad or}\\
      y^2 &= \zeta c_2 (t_2^2+1) \cdot x(x-1)(x-2\zeta t_2/(t_2+\zeta)^2),
      \end{align*}
      and the curve 
      $y^2 =                    \zeta  x(x-1)(x-2\zeta t_1/(t_1+\zeta)^2)$ is isogenous to either 
      \begin{align*}
      y^2 &=                    \zeta  x(x-1)(x+2\zeta t_2/(t_2-\zeta)^2) \text{\quad or}\\
      y^2 &= \zeta c_2 (t_2^2+1) \cdot x(x-1)(x+2\zeta t_2/(t_2-\zeta)^2).
      \end{align*}
\end{enumerate}
Then $Z_1$ and $Z_2$ are doubly isogenous.  
\end{proposition}

\begin{proof}
Since $\Jac(Z_i) \sim E_i^2$, the assumption that $E_1 \sim E_2$ implies that
$\Jac(Z_1)  \sim \Jac(Z_2)$. By Lemma~\ref{lem:family1}, the elliptic curves in
Orbit~1 for $Z_1$ are isogenous to those in Orbit 1 for $Z_2$, and similarly for
Orbits~2A and~2B. The Orbit 2C curves for $Z_1$ and $Z_2$ are isomorphic,
because each curve is the twist of $y^2 = x^3 - x$ determined by $c_i(t_i^2+1)$,
and the product of these two factors is a square. Therefore, in order for $Z_1$
and $Z_2$ to be doubly isogenous, it suffices that the elliptic curves in
Orbits~4A and 4B for $Z_1$ are, in some order, isogenous to the Orbit~4A and~4B
curves for $Z_2$. This is equivalent to conditions (1) and (2).
\end{proof}

\subsubsection{The second family}
In Table~\ref{tab:orbits}, the four elliptic curves in Orbit 4A are not
necessarily isomorphic to one another over the base field, because there are two
possible values for the factor $d$ that determines the twist of the curve. We
will refer to the first two curves (with $d = \zeta$) as the ``first pair'' of
Orbit 4A, and the last two curves (with $d = \zeta c (t^2 + 1)$) as the
``second pair.'' Likewise, we refer to the $d=\zeta$ curves in Orbit 4B as the
first pair of that orbit, and the $d = \zeta c (t^2 + 1)$ curves as the second
pair of that orbit.

As in the preceding subsection, for $i = 1$ and $i = 2$ we let $Z_i$ and $E_i$
be the curves given by~\eqref{eq:zi} and~\eqref{eq:ei}. We will consider the
following relation between $t_1$ and $t_2$: 
\begin{equation}
\label{Esecondfam}    
(t_1 - \zeta)^2 (t_2 - \zeta)^2 = -8 t_1 t_2.
\end{equation}
\goodbreak
\begin{lemma} \label{lem:family2}
Suppose \eqref{Esecondfam} holds. Then\textup{:}
\begin{enumerate}
\item the elliptic curve in Orbit 1 for $Z_1$ and the elliptic curve in Orbit~1
      for $Z_2$ are related by a degree-$2$ isogeny over $K$\textup{;}
\item if $c_1c_2(t_1^2+1)(t_2^2+1)\in K^{\times2}$, then the elliptic curves in 
      Orbit 2C for $Z_1$ are isomorphic over $K$ to the elliptic curves in 
      Orbit~2C for $Z_2$\textup{;}
\item if $c_1 t_1 (t_1^2 + 1) \in K^{\times2}$, then the elliptic curves in 
      Orbit 2A for $Z_{1}$ are isomorphic over $K$ to the first pair of elliptic
      curve in Orbit 4A of $Z_{2}$, and the elliptic curves in  Orbit 2B for 
      $Z_{1}$ are isomorphic over $K$ to the first pair of elliptic curves in 
      Orbit 4B of $Z_{2}$\textup{;}
\item if $c_1 c_2 t_1 (t_1^2 + 1) (t_2^2 + 1)\in K^{\times2}$, then the elliptic
      curves in  Orbit 2A for $Z_{1}$ are isomorphic over $K$ to the second pair
      of elliptic curve in Orbit 4A of $Z_{2}$, and the elliptic curves in 
      Orbit~2B for $Z_{1}$ are isomorphic over $K$ to the second pair of 
      elliptic curves in Orbit 4B of~$Z_{2}$\textup{;}
\item the statements obtained from \textup{(3)} and \textup{(4)} by
      interchanging the roles of $Z_1$ and $Z_2$ also hold. 
\end{enumerate}
Furthermore, if $K$ is the algebraic closure of the function field $\QQ(t)$ and
$t_1 = t$, then there are no other isogenies among the orbits associated to
$Z_1$ and~$Z_2$.
\end{lemma}

\begin{proof}
We leave the proof to the reader, because it is essentially the same as the 
proof of Lemma~\ref{lem:family1} and is mostly straightforward. The only 
non-obvious details involved in the proof of the numbered statements are that if
$t_1, t_2\in K$ satisfy~\eqref{Esecondfam}, then $\zeta t_1 t_2$ and 
$t_1(t_1^2-1)$ and $t_2(t_2^2-1)$ are all squares in~$K$. The first of these is
a square because of~\eqref{Esecondfam} and the fact that
$-8\zeta = (2-2\zeta)^2$; the second is a square because~\eqref{Esecondfam} can
be rewritten as 
\[
t_1(t_1^2 - 1) = (1 + \zeta)^2 t_1^2 (t_2 + \zeta)^2 / (t_2 - \zeta)^2;
\]
and the third is a square by symmetry.

The final statement can be proven by taking $p = 421$, $\zeta = 29\in\FF_p$,
$t_1 = 19\in \FF_p$, $t_2 = 204\in\FF_p$, and $c_1 = c_2 = 1\in\FF_p$, and
comparing traces over $\FF_{p^{12}}$ as in the end of the proof of
Lemma~\ref{lem:family1}
\end{proof}

\begin{proposition} \label{prop:family2t}
Let $Z_1$ and $Z_2$ be curves with $D_4$-action given as in~\eqref{eq:zi} by
values of $t_i$ that satisfy~\eqref{Esecondfam} and with 
$c_i = \zeta (t_i^2 + 1)$, and such that $t_2$ and $\zeta t_1$ are squares in
$K^\times$. For $i=1,2$, let $E_i$ be the elliptic curve given by~\eqref{eq:ei}.
Suppose that\textup{:}
\begin{enumerate}
\item $E_1$ and $E_2$ are isogenous to one another\textup{;}
\item the first pair of curves in Orbit 4A of $Z_1$ are isogenous to the second
      pair of curves in Orbit 4A for $Z_2$\textup{;} and 
\item the first pair of curves in Orbit 4B of $Z_1$ are isogenous to the second
      pair of curves in Orbit 4A for $Z_2$.
\end{enumerate}
Then $Z_1$ and $Z_2$ are doubly isogenous.
\end{proposition}

\begin{remark} \label{rmk:propchoices}
This proposition follows from making choices that allow us to apply certain
statements from Lemma~\ref{lem:family2}; in particular, we make choices that
imply that $c_1 t_1 (t_1^2 + 1)$ and $c_1 c_2 t_2 (t_1^2 + 1) (t_2^2 + 1)$ are
squares. Other choices would lead to variations of
Proposition~\ref{prop:family2t}.
\end{remark}

\begin{proof}[Proof of Proposition~\textup{\ref{prop:family2t}}]
From Lemma~\ref{lem:family2}(1), the curves in Orbit 1 for $Z_1$ and $Z_2$ are
isogenous to one another. From Lemma~\ref{lem:family2}(3), the curves in Orbits
2A and 2B for $Z_1$ are isogenous to the first pairs of Orbits 4A and 4B for
$Z_2$, respectively, and from Lemma~\ref{lem:family2}(5) applied to (4), the 
curves in Orbits 2A and 2B for $Z_2$ are isogenous to the second pairs of 
Orbits~4A and 4B for $Z_1$, respectively. 

Since $\zeta c_1 (t_1^2 + 1)$ and $\zeta c_2 (t_2^2 + 1)$ are both squares, 
$c_1 (t_1^2 + 1)$ and $c_2 (t_2^2 + 1)$ are in the same square class in
$K^\times$. Since the first pair and second pairs of Orbits 4A for $Z_i$ are
twists of one another by $c_i (t_i^2 + 1)$, the hypothesis \textup{(2)} of the
proposition implies that the second pair of curves in Orbit 4A of $Z_1$ are
isogenous to the first pair of curves in Orbit 4A of $Z_2$. Thus, by the 
preceding paragraph, the curves in Orbit 2A for $Z_1$ are isogenous to the 
curves in Orbit 2A for $Z_2$. Similarly, the first pair and second pairs of
Orbits 4B for $Z_i$ are twists of one another by $c_i (t_i^2 + 1)$ and the same
argument shows that the curves in Orbit 2B for $Z_1$ are isogenous to the curves
in Orbit 2B for $Z_2$.

Again as $c_1 (t_1^2 + 1)$ and $c_2 (t_2^2 + 1)$ are in the same square class in
$K^\times$, Lemma~\ref{lem:family2}(2) implies that the curves in Orbit 2C for
$Z_1$ are isomorphic to the curves in Orbit 2C for $Z_2$. As $E_1\sim E_2$, we 
conclude that $Z_1$ and $Z_2$ are doubly isogenous.
\end{proof}

As in the preceding subsection, we can apply any of eight linear fractional
transformations to $t_1$ and to $t_2$ in the relation given
by~\eqref{Esecondfam} to get another family that satisfies a lemma similar to
Lemma~\ref{lem:family2}. Only four of these families are distinct. Multiplying
the polynomials defining these four families together, we find a relation that
can be expressed in terms of the invariants $I_1$ and $I_2$ of $Z_2$ and $Z_2$:
\begin{equation}
\label{eq:family2}
    I_1 I_2 = 16.
\end{equation}

\begin{definition}
We say that two curves $Z_1$ and $Z_2$ with $D_4$-action are
\emph{in the second family} if their invariants satisfy~\eqref{eq:family2}.
\end{definition}

\begin{remark}
Equation \eqref{eq:family2} defines a genus-$0$ curve, which can be 
parametrized as
\[ I_1 = \frac{4(1-z)}{1+z}, \quad I_2 = \frac{4(1+z)}{1-z}.\]
Under this parametrization, the involution swapping $I_1$ and $I_2$ corresponds
to ${z\leftrightarrow -z}$.
\end{remark}

\begin{remark}
Proposition~\ref{prop:intrinsiccover} gives an interpretation of the first
family that can be stated in terms of the isomorphism classes of the curves,
without reference to the choices of $t_1$ and $t_2$ that we make to write down 
the curves; this is possible because the orbits involved in 
Lemma~\ref{lem:family1} are exactly the orbits contained in the Prym variety of 
a cover that can be defined independently of the choices of $t_i$. There is no 
straightforward analog of Proposition~\ref{prop:intrinsiccover} for the second 
family.
\end{remark}

\subsubsection{The third and fourth families}
We found two further families of pairs of curves where there are more isogenies
between the associated elliptic curves than expected. They produce fewer doubly
isogenous curves than the preceding two families, so here we just summarize the
results. We also simplify the exposition by assuming in this section that the
base field $K$ is algebraically closed, so that we do not have to worry about
twists.

Let $Z_1$ and $Z_2$ be genus-$2$ curves with $D_4$-action and with invariants
$I_1$ and $I_2$. We say that $Z_1$ and $Z_2$ are \emph{in the third family} if
we have
\begin{equation} \label{eq:family3}
    I_1^2 I_2^2 - 3\cdot2^8 I_1 I_2 + 2^{12} I_1 +2^{12} I_2 = 0.
\end{equation}
If $Z_1$ and $Z_2$ are in the third family, then
\begin{itemize}
\item the curves in Orbit 1 for $Z_1$ and $Z_2$ are $4$-isogenous to one
      another;
\item the curves in either Orbit 4A or Orbit 4B for $Z_1$ are $2$-isogenous to 
      the curves in either Orbit 4A or Orbit 4B for $Z_2$.
\end{itemize}
Note that \eqref{eq:family3} defines a curve of genus~$0$, which can be 
parametrized by
\[ I_1 = -32 z (z + 1)/(z - 1)^2, \quad I_2 = -32 z (z - 1)/(z + 1)^2.\]
Under this parametrization, the involution swapping $I_1$ and $I_2$ corresponds
to $z\leftrightarrow -z$.

We say that $Z_1$ and $Z_2$ are \emph{in the fourth family} if their invariants
satisfy
\begin{equation}
    \label{eq:family4}
(I_2^2 + 2^4 I_1^2 I_2 - 3\cdot 2^4 I_1 I_2 + 2^8 I_1 ) 
(I_1^2 + 2^4 I_1 I_2^2 - 3\cdot 2^4 I_1 I_2 + 2^8 I_2 )
= 0.
\end{equation}
Suppose $I_1$ and $I_2$ satisfy the first factor in this expression. Then
\begin{itemize}
\item the curves in Orbit 1 for $Z_1$ and $Z_2$ are $2$-isogenous to one
      another;
\item one of the following holds: 
    \begin{itemize}
    \item the curve $E_1$ is isomorphic to the Orbit 4A curves for $Z_2$ and the
          Orbit 2A curves for $Z_1$ are $2$-isogenous to the Orbit 4B curves 
          for~$Z_2$;
    \item the curve $E_1$ is isomorphic to the Orbit 4B curves for $Z_2$ and the
          Orbit 2A curves for $Z_1$ are $2$-isogenous to the Orbit 4A curves 
          for~$Z_2$;
    \item the curve $E_1$ is $2$-isogenous to the Orbit 4A curves for $Z_2$ and 
          the Orbit 2B curves for $Z_1$ are $2$-isogenous to the Orbit 4B curves
          for~$Z_2$;
    \item the curve $E_1$ is $2$-isogenous to the Orbit 4B curves for $Z_2$ and
          the Orbit 2B curves for $Z_1$ are $2$-isogenous to the Orbit 4A curves
          for $Z_2$.
    \end{itemize}
\end{itemize}
If $I_1$ and $I_2$ satisfy the second factor in~\eqref{eq:family4}, then the
roles of $Z_1$ and $Z_2$ in the above list are reversed. Each factor 
in~\eqref{eq:family4} defines a curve of genus~$0$.

\subsection{Intersections of families}

Under mild restrictions on the field $K$, we will produce genus-$2$ curves $Z_1$
and $Z_2$ with $D_4$-action that are very close to being doubly isogenous. The 
invariants $I_1$ and $I_2$ of these curves are the two roots of 
$x^2 - (47/16)x + 16$. Since $I_{1} I_{2} = 16$, this pair of curves lies in the
second family; since 
\[
I_1^2 I_2 + I_1 I_2^2 - 3I_1 I_2 + 1 =  16 (I_1 + I_2)  - 47 = 0,
\] 
the pair also lies in the first family; and since
\[
I_1^2 I_2^2 - 3\cdot2^8 I_1 I_2 + 2^{12} I_1 +2^{12} I_2 
= 16^2 - 3\cdot 2^8 \cdot 16 + 2^{12}(47/16) = 0,
\]
the pair lies in the third family as well.

\begin{proposition}
\label{Prop:IntersectionFamilyIsogenousFactors}
Let $K$ be a field in which $-1$, $2$, and $-7$ are nonzero squares, and let 
$\zeta\in K$ satisfy $\zeta^2 = -1$. In $K$, let
\[
t_1 = \frac{( 1 + \zeta)(3 + \zeta\sqrt{-7})}{2} \text{\quad and \quad}
t_2 = \frac{(-1 + \zeta)(3 + \zeta\sqrt{-7})}{2}.
\]
For $i=1,2$, let $c_i=1$, and let $Z_i$ and $E_i$ be defined by~\eqref{eq:zi}
and~\eqref{eq:ei}. Then the invariants of $Z_1$ and $Z_2$ are the two roots of
$x^2 - (47/16)x + 16$, and if $E_1$ is isogenous to~$E_2$, the curves $Z_1$ and
$Z_2$ are doubly isogenous.
\end{proposition}

\begin{proof}
An easy computation verifies the statement about the invariants of $Z_1$ 
and~$Z_2$.

We check that $t_2 = \zeta t_1$ and that $(t_1^2 + 1)(t_2^2 + 1)$ is equal to
the square of ${\sqrt{2}^3 (3 + \zeta\sqrt{-7})}$. Since $c_1c_2 = 1$ is also a
square, the hypotheses of Lemma~\ref{lem:family1} are satisfied.

We also check that~\eqref{eq:family2} holds, and that 
\begin{align*}
c_1 t_1 (t_1^2 + 1) 
  &= \sqrt{2}^{-2}(2 + 5\zeta - 2\sqrt{-7} + \zeta\sqrt{-7})^2 
     \text{\quad and}\\
c_2 t_2 (t_2^2 + 1) 
  &= \sqrt{2}^{-2}(5 + 2\zeta - \sqrt{-7} + 2\zeta\sqrt{-7})^2.
\end{align*}
As we already noted, $(t_1^2 + 1)(t_2^2 + 1)$ is a square, so the hypotheses of 
statements (1),(2), and (3) of Lemma~\ref{lem:family2} hold, as does the
hypothesis of the variation of the lemma's statement (3) obtained by
interchanging the roles of $Z_1$ and $Z_2$.

Finally, we note that the first pair and the second pair of Orbit 4A for $Z_1$
are twists of one another by $c_1 (t_1^2 + 1)$, while the first pair and the
second pair of Orbit 4A for $Z_2$ are twists of one another by 
$c_2 (t_2^2 + 1)$; these two twisting factors lie in the same square class in
$K^\times$. The analogous statement holds for the first and second pairs of
Orbit 4B for $Z_1$ and for $Z_2$.

Combining the conclusions of Lemma~\ref{lem:family1} and Lemma~\ref{lem:family2}
with this last observation, it is straightforward to verify that $Z_1$ and $Z_2$
are doubly isogenous.
\end{proof}

\begin{remark}
It is easy to check that the only pair $(I_1,I_2)$ of nonzero elements of $K$
that satisfies~\eqref{eq:family1} and~\eqref{eq:family2} is the pair from
Proposition~\ref{Prop:IntersectionFamilyIsogenousFactors}. This pair is also the
only pair to satisfy both~\eqref{eq:family2} and~\eqref{eq:family3}. There are
other pairs that satisfy the defining equations of more than one of the four
families, but in characteristic $0$ the curves with those invariants do not have
as many isogeny factors in common as the curves in 
Proposition~\ref{Prop:IntersectionFamilyIsogenousFactors}.
\end{remark}

\begin{remark}
As we noted in Remark~\ref{rem:nonisomorphic}, if two curves over a finite 
field are Galois conjugates of one another, they are necessarily doubly
isogenous. We check that the values of $s_i$ (see Equation~\eqref{Estformula}) 
for the two curves in Proposition~\ref{Prop:IntersectionFamilyIsogenousFactors}
are $s_1 = 6\sqrt{-7}$ and $s_2 = -6\sqrt{-7}$. Thus, if $K$ is finite and $-7$
is not a square in its prime field, the curves in the proposition are
automatically doubly isogenous for an unsurprising reason.
\end{remark}

\begin{example}
In $K = \FF_{113}$, take $\zeta = 15$ and $\sqrt{-7} = 28$, and apply
Proposition~\ref{Prop:IntersectionFamilyIsogenousFactors}. We find that 
$t_1 = 107$ and $t_2 = 23$, that $s_1 = 55$ and $s_2= 58$, and that the elliptic
curves $E_1$ and $E_2$ both have trace $6$, so that $E_1 \sim E_2$. This gives
an example where the curves in 
Proposition~\ref{Prop:IntersectionFamilyIsogenousFactors} are not Galois 
conjugates of one another, but are doubly isogenous.
\end{example}

\begin{remark}
Let $E_1$ be the elliptic curve $y^2 = (x + 1)(x^2 + 6\sqrt{-7}x + 1)$ over the 
field $K = \QQ(\sqrt{-7})$ and let $E_2$ be its conjugate over $\QQ$. If
$\frakp$ is a prime of $K$ such that the reductions of $E_1$ and $E_2$ modulo
$\frakp$ are isogenous, then over an extension of the residue field of~$\frakp$,
the two curves $y^2 = (x^2 + 1)(x^4 \pm 6\sqrt{-7} x^2 + 1)$ will be doubly
isogenous. By \cite[Theorem~1.1]{Charles2018}, there are infinitely many such
primes. However, the two curves we obtain from such a prime will be conjugate to
one another --- and hence, will give an uninteresting example --- exactly when
$\frakp$ is a prime of degree~$2$. 

This leads to the following question: Are there infinitely many degree-$1$ 
primes $\frakp$ of $\QQ(\sqrt{-7})$ such that the reductions of $E_1$ and $E_2$
modulo $\frakp$ are isogenous? We suspect that the answer is yes, because if
$\frakp$ lies over $p$ then heuristically there is roughly one chance out of
$\sqrt{p}$ that they will be isogenous by chance, and the sum of $1/\sqrt{p}$
diverges. Unfortunately, we do not know how to prove that there are infinitely
many such primes.
\end{remark}

\section{Heuristics for the families of coincidences}
\label{ss:heuristicsfamily}

In the preceding section, we identified four families of pairs $(Z_1,Z_2)$ of 
genus-$2$ curves with $D_4$-action where there are unexpected isogenies among a 
number of the elliptic curves that appear in the decomposition (up to isogeny)
of the Jacobians $\Jac(\Ztilde_1)$ and $\Jac(\Ztilde_2)$. In this section, we
formulate heuristics for the expected number of doubly isogenous pairs over a
finite field that occur in these families.

\subsection{Counting doubly isogenous pairs in families}

First, we introduce notation to keep track of the number of doubly isogenous
pairs in each family; see also Definition~\ref{def:delta}.

\begin{definition}
\label{def:deltan}
We consider unordered pairs $\{Z_1, Z_2\}$ of doubly isogenous curves 
over~$\Fq$, where $Z_1$ and $Z_2$ are genus-$2$ curves with $D_4$-action and all
Weierstrass points rational, and where $Z_1$ and $Z_2$ are not Galois conjugates
of one another. For $n = 1, 2, 3, 4$, let $\delta_n(q)$ be the number of these
pairs $\{Z_1,Z_2\}$ that lie in the $n$th family. Let $\delta_0(q)$ be the
number of these pairs $\{Z_1,Z_2\}$ that do \emph{not} lie in any of these four
families. Let $\delta_{1,2,3}(q)$ be the number of these pairs $\{Z_1,Z_2\}$
that are simultaneously in families $1$, $2$, and $3$ --- that is, the pairs
whose invariants are the two roots of $x^2 - (47/16)x + 16$.
\end{definition}

\begin{remark}
For $n > 0$, each $\delta_n(q)$ counts doubly isogenous pairs whose invariants
satisfy one of the four equations~\eqref{eq:family1}, \eqref{eq:family2},
\eqref{eq:family3}, or \eqref{eq:family4}. We do not demand that the doubly
isogenous curves come from values of $t$ that are related to one another as in,
for example, Lemma~\ref{lem:family1} or  Lemma~\ref{lem:family2}.
\end{remark}

We saw in Section~\ref{sec:heuristicdoublenaive} that Na\"{\i}ve
Heuristic~\ref{FH:DoublyIsogenous} did not seem to reflect the data that we had
collected. Here we present another heuristic which better reflects the data. In 
each family, let $m$ be the number of pairs of elliptic curves required to be
isogenous to ensure double isogeny of two curves in that family; for example,
Proposition~\ref{prop:family1t} shows that $m=3$ for family 1. Then we model the
double isogeny class of a curve in that family as a $m$-tuple of independent
elliptic curves. Given this, as in Section~\ref{ss:initialheuristic}, one can
compute the expected number of doubly isogenous pairs in the given families
under this heuristic. By abuse of notation, we will denote this by
$\EE(\delta_i(q))$, even though for fixed $q$, the integer $\delta_i(q)$ is a
fixed value and not a random variable.

\begin{heuristic}
\label{H:families}
The following are reasonable estimates for the ``expected value'' of
$\delta_n(q)$ for prime powers $q\equiv 1\bmod 4$\textup{:}
\begin{alignat*}{6}
                         \EE(\delta_0(q))       &\asymp 1/q, 
&                 \qquad \EE(\delta_1(q))       &\asymp 1/\sqrt{q}, 
&                        \EE(\delta_2(q))       &\asymp 1/\sqrt{q},\\
                         \EE(\delta_3(q))       &\asymp 1/\sqrt{q}, 
&                        \EE(\delta_4(q))       &\asymp 1/q^{3/2},  
& \quad \text{and}\quad  \EE(\delta_{1,2,3}(q)) &\asymp 1/\sqrt{q}.
\end{alignat*}
Combining these values, we expect that $\EE(\delta(q))\asymp 1/\sqrt{q}$.
\end{heuristic}

\begin{proof}[Justification]
First we consider $\delta_0(q)$. If we assume that there are no families of
unexpected coincidences other than the four presented in
Section~\ref{ssec:families}, then the justification of Na\"{\i}ve
Heuristic~\ref{FH:DoublyIsogenous} applies to the pairs $\{Z_1,Z_2\}$ that are
not in these four families; this suggests that the expected value of
$\delta_0(q)$ is $\Theta(1/q)$.

Next we consider $\delta_1(q)$, and in particular we look at the pairs of 
curves $(Z_1,Z_2)$ which can be written with $t_2 = \zeta t_1$ and $c_2 = c_1$. 
There are roughly $q$ such ordered pairs. By Proposition~\ref{prop:family1t}, in
order for such a pair to be doubly isogenous, it is sufficient for three
coincidences to hold: $E_1$ and $E_2$ should be isogenous, and one of the two
pairs of isogenies in items (1) and (2) of Proposition~\ref{prop:family1t}
should hold. We model each of these coincidences as asking that two random
elliptic curves lie in the same isogeny class, which happens with probability
$\Theta(1/\sqrt{q})$. Thus, we expect to find on the order of 
$q / (\sqrt{q})^3 = 1/\sqrt{q}$ doubly isogenous pairs with $t_2 = \zeta t_1$
and $c_2 = c_1$.

There are other relations between $t_1$ and $t_2$ that lead to
Equation~\eqref{eq:family1} holding, and again we expect $\Theta(1/\sqrt{q})$
doubly isogenous pairs that satisfy the relation. Thus, in total, we expect 
$\Theta(1/\sqrt{q})$ pairs of doubly isogenous pairs in the first family.

For $\delta_2(q)$ the argument is similar. By Proposition~\ref{prop:family2t},
if each of three pairs of elliptic curves are isogenous to one another, the
curves $Z_1$ and $Z_2$ in the proposition are doubly isogenous. Again modeling
these isogeny class collisions as occurring with probability
$\Theta(1/\sqrt{q})$, we find that we expect $\Theta(1/\sqrt{q})$ pairs of
curves $(Z_1,Z_2)$ coming from pairs of values $(t_1,t_2)$
satisfying~\eqref{Esecondfam}. 

As for the first family, there are other relations between $t_1$ and $t_2$ that
lead to Equation~\eqref{eq:family2} holding. For each such relation we again
expect $\Theta(1/\sqrt{q})$ doubly isogenous pairs that satisfy the relation.
Again, in total, we expect  $\Theta(1/\sqrt{q})$ pairs of doubly isogenous pairs
in the second family.

We skip over the third family for the moment, for reasons that will become
apparent.

For pairs $\{Z_1,Z_2\}$ in the fourth family, we need five coincidental 
isogenies in order for the curves to be doubly isogenous. This suggests that 
the expected number of such curves is $\Theta(q/q^{5/2}) = \Theta(1/q^{3/2})$.

Proposition~\ref{Prop:IntersectionFamilyIsogenousFactors} suggests that when
$-7$ is a square in $\Fq$ there is one chance out of $\sqrt{q}$ that the two
values of $t$ in the proposition give rise to a doubly isogenous pair over 
$\Fq$ that lies in the first, second, and third families. When $-7$ is a square 
in the prime field the curves in this pair are not Galois conjugates of one
another and so contribute to the value of $\delta_{1,2,3}(q)$. For other pairs
of curves whose invariants are the roots of $x^2 - (47/16)x + 1$ and that are
not Galois conjugates of one another, the likelihood of being doubly isogenous
is less than this, so in total the expected value of $\delta_{1,2,3}(q)$ is 
$\Theta(1/\sqrt{q})$.

For pairs $\{Z_1,Z_2\}$ in the third family, if we argue as above we see that we
need four coincidental isogenies in order for the curves to be doubly isogenous.
This suggests that the expected number of such curves is 
$\Theta(q/q^{4/2}) = \Theta(1/q)$. But once again our na\"{\i}ve analysis needs
revision, because clearly $\delta_3(q)\ge\delta_{1,2,3}(q)$. Thus, we take the
expected value of $\delta_3(q)$ to be $\Theta(1/\sqrt{q})$.
\end{proof}

\subsection{Comparison with data} \label{ss:comparisondata}
For $n=15,\ldots, 23$, we considered the $1024$ primes $q$ with 
$q \equiv 1 \bmod 4$ closest to $2^n$. For each such $q$, we found all 
unordered pairs $\{Z_1,Z_{2}\}$ of non-conjugate curves over $\Fq$ with
$D_4$-action and with all Weierstrass points rational for which $Z_1$ and $Z_2$
are doubly isogenous.  A given pair may appear in more than one family.
Table~\ref{Table:families} shows which of these pairs are explained by one 
(or more) of the families. 

%

\begin{table}[ht]
\centering
\begin{tabular}{rccc@{}r@{\quad}r@{\quad}r@{\quad}rc}
\toprule
    &                           & In a   & Not in a &    &    &    &    &           \\
$n$ & \quad Total\hbox to 1em{} & Family & Family   & \qquad F1 & F2 & F3 & F4 & (F1 $\cap$ F2 $\cap$ F3)\\
\midrule
15 &    820 &    586 &      234 & 366 & 222 & \pz 62 & \pz 20 &    38 \\
16 &    580 &    494 &   \pz 86 & 286 & 198 &     34 &      0 &    12 \\
17 &    407 &    318 &   \pz 89 & 192 & 138 &     24 &      0 &    18 \\
18 &    282 &    238 &   \pz 44 & 148 &  96 &     14 &      0 &    10 \\
19 &    218 &    196 &   \pz 22 & 116 &  90 &     10 &      2 &    10 \\
20 &    138 &    132 & \pz\pz 6 &  78 &  58 &      4 &      0 & \pz 4 \\
21 &    100 & \pz 90 &   \pz 10 &  54 &  40 &      4 &      0 & \pz 4 \\
22 & \pz 58 & \pz 58 & \pz\pz 0 &  40 &  16 &      2 &      0 & \pz 0 \\
23 & \pz 42 & \pz 40 & \pz\pz 2 &  20 &  20 &      0 &      0 & \pz 0 \\
 \bottomrule
\end{tabular}

\vskip 2ex
\caption{Data for doubly isogenous curves. For each $n$, column 2 contains the
total number of (unordered) pairs of doubly isogenous curves over $\Fq$ for the
$1024$ primes $q\equiv 1\bmod 4$ closest to $2^n$. The $3$rd (resp.\ $4$th)
column contains the number of these in (resp.\ not in) at least one family. The
remaining columns contain the number for each family.}
\label{Table:families}
\end{table}

As predicted by Heuristic~\ref{H:families}, increasing $n$ by \emph{two} appears
to roughly halve the total number of pairs, as well as the pairs in family $1$
or family $2$ (and possibly in family $3$ and in the intersection of the first
three families, although it is harder to tell because the numbers are smaller).
In contrast, increasing $n$ by \emph{one} appears to roughly halve the number of
pairs coming from no family. This is as expected as 
$\Theta(q^{-1/2}) = \Theta( 2^{-n/2})$ and $\Theta(q^{-1}) = \Theta(2^{-n})$ for
primes $q$ near $2^n$. The numbers for the fourth family drop off too rapidly to
easily determine the rate of decline, but our heuristics do at least predict
that the fourth family will decrease the fastest.

\subsubsection*{Thanks\textup{:}}
This work was supported by a grant from the Simons Foundation (546235) for the
   collaboration `Arithmetic Geometry, Number Theory, and Computation', through 
   a workshop held at ICERM.
Booher was partially supported by the Marsden Fund Council administered by the
   Royal Society of New Zealand. 
Li was partially funded by the Simons collaboration on `Arithmetic Geometry, 
   Number Theory, and Computation'.
Pries was partially supported by NSF grant DMS-19-01819.  
Springer was partially supported by National Science Foundation Awards 
   CNS-2001470 and CNS-1617802.
   
We thank Bjorn Poonen and Felipe Voloch for helpful conversations.


\bibliography{references}
\bibliographystyle{hplaindoi}

\end{document}